\documentclass[leqno]{amsart}
\usepackage{amsfonts,amssymb,amsmath,amsgen,amsthm}
\usepackage{hyperref}

\theoremstyle{plain}
\newtheorem{theorem}{Theorem}
\newtheorem{definition}{Definition}[section]
\newtheorem{assumption}[definition]{Assumption}
\newtheorem{lemma}[definition]{Lemma}
\newtheorem{corollary}[definition]{Corollary}
\newtheorem{proposition}[definition]{Proposition}

\newtheorem{remark}[definition]{Remark}

\usepackage{color,ulem,textcomp}

\def\refeq#1{~\eqref{#1}}

\def\longformule#1#2{
\displaylines{ \qquad{#1} \hfill\cr \hfill {#2} \qquad\cr } }
\def\sumetage#1#2{
\sum_{\scriptstyle {#1}\atop\scriptstyle {#2}} } \font\tenms=msbm10
 \font\tenms=msbm10

\font\sevenms=msbm7 \font\fivems=msbm5
\newfam\bbfam \textfont\bbfam=\tenms \scriptfont\bbfam=\sevenms
\scriptscriptfont\bbfam=\fivems

\def\build#1_#2^#3{\mathrel{
\mathop{\kern 0pt#1}\limits_{#2}^{#3}}}

\newcommand{\ds}{\displaystyle}
\newcommand{\beq}{\begin{eqnarray}}
\newcommand{\eeq}{\end{eqnarray}}
\newcommand{\bq}{\begin{equation}}
\newcommand{\eq}{\end{equation}}
\newcommand{\beqn}{\begin{eqnarray*}}
\newcommand{\eeqn}{\end{eqnarray*}}

\let\al=\alpha

\let\lam=\lambda

\let\wt=\widetilde

\def\DD{\mathop{\bf D\kern 0pt}\nolimits}

\def\SS{\mathop{\bf S\kern 0pt}\nolimits}
\def\ZZ{\mathop{\bf Z\kern 0pt}\nolimits}
\def\TT{\mathop{\bf T\kern 0pt}\nolimits}
\def\virgp{\raise 2pt\hbox{,}}
\def\cdotpv{\raise 1pt\hbox{ ;}}

\def\beq{\begin{equation}}
\def\eeq{\end{equation}}

\def\cdotv{\raise 2pt\hbox{,}}

\def\H{\mathop{\mathbb H\kern 0pt}\nolimits}
\def\R{\mathop{\mathbb R\kern 0pt}\nolimits}

\def\N{{\mathbf N}}% nonnegative integers
\def\virgp{\raise 2pt\hbox{,}}

\def\({\left(}
\def\){\right)}
\def\<{\left\langle}
\def\>{\right\rangle}

\def\ge{\geqslant}

\numberwithin{equation}{section}

\begin{document}
\title[Dispersive estimates and Schr\"odinger operators on step 2 Stratified Lie Groups]{Dispersive estimates for the Schr\"odinger operator on step 2 Stratified Lie groups}
\author[H. Bahouri]{Hajer Bahouri}
\address[H. Bahouri]{LAMA UMR CNRS 8050,
Universit\'e Paris EST\\ 61, avenue du G\'en\'eral de Gaulle\\
94010 Cr\'eteil Cedex\\ France }
\email{hbahouri@math.cnrs.fr}
\author[C. Fermanian]{Clotilde~Fermanian-Kammerer}
\address[C. Fermanian]{LAMA UMR CNRS 8050,
Universit\'e Paris EST\\ 61, avenue du G\'en\'eral de Gaulle\\
94010 Cr\'eteil Cedex\\ France}
\email{clotilde.fermanian@u-pec.fr}
\author[I. Gallagher]{Isabelle Gallagher}\address[I. Gallagher]%
{Institut de Math{\'e}matiques UMR 7586 \\
      Universit{\'e} Paris Diderot (Paris 7) \\
 B\^atiment Sophie Germain  \\
    Case 7012\\
 75205 PARIS Cedex 13     \\
      France}
\email{gallagher@math.univ-paris-diderot.fr}

\begin{abstract}

The present paper is dedicated to the proof of
dispersive estimates on stratified Lie groups of step 2, for the linear 
   Schr\"odinger equation involving   a  sublaplacian. 
   It turns out that the propagator  behaves like a wave operator   on a space of the  same dimension~$p$ as  the center of the group, and like a Schr\"odinger operator on a space of the same dimension~$k$ as  the radical of the canonical skew-symmetric form, which suggests a decay rate $|t|^{-{k+p-1\over 2}}$.  
   In this article, we identify 
     a property of the canonical skew-symmetric form under which we establish optimal dispersive estimates with this rate.  The relevance of this property is discussed through several examples.
   \end{abstract}
\thanks{}
\maketitle

%\tableofcontents

\section{Introduction}\label{intro}
\subsection{Dispersive inequalities}
\label{sec:up}

Dispersive inequalities for evolution  equations (such as Schr\"odinger and wave equations) play a decisive role in the study of  semilinear and quasilinear problems which appear in numerous physical applications. Proving dispersion   amounts to establishing a decay estimate for the $L^\infty$ norm  of the solutions of these equations at time $t$ in terms of some negative power of $t$ and the~$L^1$ norm  of the data.  In many cases, the main step in the proof of this decay in time relies on the application of a stationary  phase theorem on an (approximate) representation of the solution.
Combined with an abstract functional analysis argument known as the $TT^*$-argument, dispersion phenomena yield a range of estimates involving space-time Lebesgue norms.   Those inequalities, called Strichartz estimates, have proved to be powerful in the study of nonlinear equations (for instance one can consult \cite{bcd} and the references therein).

\bigskip

\noindent In the $\R^d$ framework, dispersive inequalities have a long history beginning with the articles of Brenner~\cite{brenner1}, Pecher  \cite{P1},  Segal \cite{segal} and Strichartz \cite{strichartz}.  They were subsequently developed by various authors, starting with the paper of Ginibre and Velo~\cite{ginibrevelo} (for a detailed bibliography, we refer to~\cite{keeltao, tao}   and the references therein).
In \cite{bgx}, the authors generalize  the dispersive estimates for the wave equation to the Heisenberg group $\H^d$ with an optimal rate of decay of order $ | t   |^{- 1/2}$ (regardless of the dimension~$d$)
and prove    that no dispersion    occurs for the Schr\"odinger  equation.  In \cite{hiero},  optimal results are proved for the time behavior of the Schr\"odinger and wave equations on   H-type groups: if~$p$ is   the dimension of the center of the H-type group, the author establishes sharp dispersive inequalities for the wave equation solution (with a decay rate of~$ | t   |^{- p/2}$) as well as for the Schr\"odinger  equation solution (with a~$ | t   |^{-(p-1)/2}$ decay). Compared with the $\R^d$ framework, there is an  exchange in the rates of decay between the wave and the Schr\"odinger equations.

\bigskip

\noindent   Strichartz estimates in  other settings   have been obtained in a number of works. One can first cite   various results dealing  with variable coefficient   operators (see  for instance \cite{kapitanski, smith1}) or studies concerning  domains such as \cite{blp,ilp,ss}. One can also refer to  the result concerning  the full Laplacian on the Heisenberg group in  \cite{furioli2},    works in the framework of the real hyperbolic spaces  in   \cite{AP, banica, tataru}, or in the framework of compact and noncompact manifolds in \cite{A, bd,bgt}; finally one can mention   the quasilinear framework studied in  \cite{bch, bch2, kr,st}, and the references therein. 

\bigskip

\noindent In this paper our   goal is to establish optimal dispersive estimates for the solutions of the Schr\"odinger equation on   {$2$-step  stratified}  Lie groups.  We shall emphasize in particular the  key  role played by the canonical skew-symmetric form   in determining the rate of decay of the solutions.  It turns out that the Schr\"odinger  propagator on~$G$ behaves like a wave operator   on a space of the same dimension as the center of  $G$, and like a Schr\"odinger operator on a space of the same dimension as the radical of the canonical skew-symmetric form associated with the dual of the center. This unusual behavior of the Schr\"odinger  propagator in the case of Lie algebras whose canonical skew-symmetric form is degenerate (known as Lie algebras which are not 
MW, see~\cite{MooreWolf},~\cite{MR} for example)  makes the analysis of  the  explicit representations of the solutions  tricky  and gives rise to uncommon dispersive estimates. It will also appear from our analysis that the optimal rate of decay is not always in accordance with the dimension of the center: we shall exhibit examples of   $2$-step  stratified  Lie groups with center of any dimension   and for which no dispersion    occurs for the Schr\"odinger  equation. We shall actually highlight that the optimal rate of decay in the dispersive estimates for solutions  to the Schr\"odinger equation is rather related to the properties of the canonical skew-symmetric form.

\bigskip

\subsection{Stratified Lie groups}\label{gradedliegroup}
   Let us recall here some basic facts about  stratified Lie groups (see~\cite{corwingreenleaf,folland, follandstein,% lemarie,
   steinweiss}  and the references therein for further details).   A connected, simply connected nilpotent Lie group $G$ is
  said stratified if its left-invariant Lie algebra~${\mathfrak g}$ (assumed  real-valued and  of finite dimension~$n$)    is  endowed with a vector space decomposition
  $$
  \displaystyle
  {\mathfrak g}=   \oplus_{1\leq k\leq \infty} \, {\mathfrak g}_k \, ,
  $$
 where all but finitely many of the ${\mathfrak g}_k'$s are $\{0\}$,   such that $[{\mathfrak g}_1,{\mathfrak g}_{k}]= {\mathfrak g}_{k+1}$. 
 If there are~$p$ non zero~${\mathfrak g}_k$ then the group is said of step~$p$.
 Via the exponential map  
 $$
 {\rm exp} :  {\mathfrak g} \rightarrow G 
 $$ which is in that case a diffeomorphism from ${\mathfrak g}$ to $G$, one identifies $G$ and ${\mathfrak g}$. It turns out that under this identification, the  group law on $G$ (which is generally not commutative) provided by the Campbell-Baker-Hausdorff formula, $(x,y)  \mapsto x \cdot y  $ is a polynomial map.  In the following we shall denote by~$\mathfrak z$ the center of~$G$ which is simply the last non zero~${\mathfrak g}_k$ and write 
 \begin{equation}\label{eq:Gdec}
  G =  \mathfrak v \oplus \mathfrak z \, ,
  \end{equation}
 where $\mathfrak v$ is any subspace of $G$ complementary to $\mathfrak z$.

    \medskip

 \noindent  The group $G$ is endowed with a smooth left invariant measure $\mu(x)$, the  Haar measure, induced by the Lebesgue measure on ${\mathfrak g}$ and  which satisfies the fundamental translation invariance property:
  $$
  \forall f  \in L^1(G, d\mu) \, ,  \quad  \forall x  \in G \,, \quad \int_G f(y) d\mu(y)  = \int_G f(x \cdot y)d\mu(y) \, .
   $$
Note that  the convolution of two functions $f$ and $g$ on $G$ is given by
    \begin{equation}
\label{convolutiondef}
  f*g(x) :=  \int_G f(x \cdot y^{-1})g(y)d\mu(y) = \int_G f(y)g(y^{-1} \!   \cdot   x)d\mu(y)
   \end{equation}
  and  as in the euclidean case we define   Lebesgue spaces by
$$
 \|f\|_{L^p (G)}  := \left( \int_G |f(y)|^p \: d \mu (y) \right)^\frac1p \, ,
 $$
 for $p\in[1,\infty[$, with the standard modification when~$p=\infty$.\\

 \noindent    Since $G$ is stratified,   there is a natural family of dilations on ${\mathfrak g}$ defined for $t>0$ as follows: if~$X$ belongs to~$ {\mathfrak g}$, we can decompose~$X$ as~$\displaystyle X=\sum X_k$ with~$\displaystyle  X_k\in {\mathfrak g}_k$, and then
   $$
   \delta_t X:=\sum t^{k} X_k \, .
   $$
 This allows to  define the dilation $\delta_t $ on the Lie group $G$ via the identification by the exponential map:
 $$\begin{array}{ccccc}
& {\mathfrak g} &\build{\rightarrow}_{}^{\delta_t} & {\mathfrak g}&\\
 {\small\rm exp}&  \downarrow& & \downarrow& {\small\rm exp}\\
  &G &\build{\rightarrow}_{ {\rm exp}\, \circ\,  \delta_t \, \circ\,  {\rm exp}^{-1}}^{}&G
  \end{array}$$
 To avoid heaviness,   we shall still denote by $\delta_t$ the map ${\rm exp}\, \circ \delta_t \, \circ {\rm exp}^{-1}$.\\

 \noindent Observe  that the action of the left invariant vector fields $X_k$, for~$X_k$ belonging to~${\mathfrak g}_k$, changes the homogeneity in the following way:
 $$
 X_k (f \circ \delta_t) = t^{k} X_k (f )\circ \delta_t \, ,
 $$
 where by definition $\displaystyle X_k (f ) (y) := \frac d {ds} f \bigl(y \cdot {\rm exp}  (s X_k)\bigr)_{|s=0}$ and   the Jacobian of the dilation $\delta_t$ is $t^Q$ where~$\displaystyle{Q:=\sum_{1\leq k\leq \infty} k\, {\rm dim} \, {\frak g}_k}$ is called the homogeneous dimension of $G$:
   \begin{equation}\label{homogenedim} \int_G f(\delta_t\,y) \,d\mu(y)  =  t^{-Q}\,\int_G f( y)\,d\mu(y) \, .
   \end{equation}
\noindent Let us also point out that there is  a natural norm~$\rho$ on~$G$ which is homogeneous in the sense
that it respects  dilations: $G\ni x\mapsto \rho(x)$ satisfies
$$
\forall x\in G,\;\;\rho(x^{-1})=\rho(x) \, ,\;\;\rho(\delta_tx)=t\rho(x) \, , \;{\rm and}\;\;\rho(x)=0\;\Longleftrightarrow\; x=0 \, .
$$
We can   define the Schwartz space~${\mathcal S}(G)$   as the set of smooth functions on~$G$ such that
 for all~$ \alpha$ in~${\mathbb N}^d$,  for all~$p$ in~${\mathbb N}, x\mapsto \rho(x)^p   {\mathcal X}^{\alpha}f(x) $ belongs to~$ L^\infty(G),
  $ where~${\mathcal X}^{\alpha}$ denotes a product of $|\alpha|$  left invariant vector fields. The Schwartz space~${\mathcal S}(G)$ has properties very similar to those of the Schwartz space~${\mathcal S}(\R^d)$, particularly density in Lebesgue spaces.
 
\subsection{The Fourier transform}\label{Fourier}  
The group $G$ being non commutative, its Fourier transform is defined by means of  irreducible unitary representations.  We devote this section to the introduction of  the basic concepts that will be needed in the sequel. From now on, we assume that $G$ is a step~2 stratified Lie group, meaning~$ \mathfrak z =  \mathfrak g_2$,  and we denote 
${\mathfrak v}={\mathfrak g}_1$ in~(\ref{eq:Gdec}).  We choose a scalar product on~${\mathfrak g}$ such that~${\mathfrak v}$ and~${\mathfrak z}$ are orthogonal.

\subsubsection{Irreducible unitary representations}\label{defirreducible} Let us fix some notation,   borrowed from~\cite{crs} (see also~\cite{corwingreenleaf} or~\cite{MR}).  For any~$\lambda \in  \mathfrak z^\star$ (the dual of the center~$  \mathfrak z$) we define a skew-symmetric bilinear form on $\mathfrak v$ by 
\begin{equation}\label{skw}
\forall \, U,V \in  \mathfrak v \, , \quad B(\lambda) (U,V):= \lambda([U,V]) \, .
 \end{equation}
 One can find a   Zariski-open subset~$\Lambda$ of~$  \mathfrak z^\star$ such that
 the number of distinct eigenvalues of~$B(\lambda)$ is maximum. We denote by~$k$ the dimension of   the radical $\mathfrak r_\lambda$ of~$ B(\lambda)$. Since~$ B(\lambda)$ is skew-symmetric, the dimension of the orthogonal complement of $\mathfrak r_\lambda$ in $\mathfrak v$ is an even number which we shall denote by~$2d$. Therefore, there exists an orthonormal basis
  $$\big (P_1(\lambda) , \dots ,P_d(\lambda),  Q_1(\lambda) , \dots ,Q_d(\lambda),R_1(\lambda),\dots,R_k(\lambda)\big )$$ 
  such that the  matrix of~$B(\lambda)$ takes the following form
  $$
\left(
\begin{array}{ccccccccc}
0 &\dots & 0& \eta_1(\lambda)& \dots& 0 & 0 & \cdots & 0 \\
\vdots &\ddots & \vdots& \vdots& \ddots& \vdots & \vdots & \ddots& \vdots \\\
0 &\dots & 0&0& \dots&  \eta_d(\lambda) & 0 & \cdots & 0\\
- \eta_1(\lambda) &\dots & 0&0& \dots& 0& 0 & \cdots & 0\\
\vdots &\ddots & \vdots& \vdots& \ddots& \vdots&\vdots& \ddots& \vdots\\
0 &\dots & - \eta_d(\lambda)&0& \dots& 0& 0 & \cdots & 0\\
0 & \dots & 0 & 0 & \dots & 0& 0 & \cdots & 0 \\
\vdots &\ddots & \vdots& \vdots& \ddots& \vdots&\vdots& \ddots& \vdots\\
0 &\dots &0&0& \dots& 0& 0 & \cdots & 0
\end{array}
\right)
\, ,$$
where each~$\eta_j(\lambda)>0 $ is smooth and   homogeneous   of degree one in~$\lambda = (\lambda_1,\dots, \lambda_p)$  and the basis vectors  are chosen to depend smoothly on~$\lambda$ in~$\Lambda$. 
Decomposing~$ \mathfrak v$ as
$$
 \mathfrak v = \mathfrak p_\lambda +  \mathfrak q_\lambda  +\mathfrak r_\lambda $$
 with 
 $$
 \begin{aligned}
  \mathfrak p_\lambda:= \mbox{Span} \, \big (P_1(\lambda) , \dots ,P_d(\lambda) \big) \, , & \quad \mathfrak q_\lambda:= \mbox{Span} \, \big (Q_1(\lambda) , \dots ,Q_d(\lambda)\, ,\quad  \mathfrak r_\lambda := &\mbox{Span} \, \big (R_1(\lambda) , \dots ,R_k(\lambda) \big)
   \end{aligned}
$$
any element $V \in \mathfrak v $ will be written in the following as~$P+Q+R$ with~$P\in  \mathfrak p_\lambda$, $Q \in  \mathfrak q_\lambda$ 
and~$R\in\mathfrak r_\lambda$.
We then introduce irreducible unitary representations of~$G$ on~$L^2( \mathfrak p_\lambda)$:
\begin{equation}\label{defpilambda}
u^{\lambda,\nu}_{X} \phi(\xi) := {\rm e}^{-i\nu( R)-i\lambda (Z+ [ \xi +\frac12 P , Q])} \phi (P+ \xi) \, , \;\lambda\in\mathfrak z^*,\;\nu\in\mathfrak r^*_\lambda \, , 
\end{equation}
for any~$x=\exp (X)\in G$ with~$X = X(\lambda,x) :=\big (P(\lambda,x),Q(\lambda,x),R(\lambda,x),Z(x) \big)$ and~$\phi \in L^2( \mathfrak p_\lambda)$. In order to shorten notation, we shall omit the dependence on~$(\lambda,x)$ whenever there is no risk of confusion.

\subsubsection{The Fourier transform}  In contrast with the euclidean case, the Fourier transform  is defined on the bundle~$\mathfrak r(\Lambda)$ above~$\Lambda$ whose fibre above $\lambda\in\Lambda$ is $\mathfrak r^* _\lambda\sim \R^k$. It is valued  in   the space of  bounded operators
on~$L^2( \mathfrak p_\lambda)$. More precisely, the Fourier transform of a function~$f$ in~$L^1(G)$ is  defined  as follows: for any~$(\lambda,\nu) \in\mathfrak r(\Lambda)$ $$
{\mathcal F}(f)(\lambda,\nu):=
\int_G f(x) u^{\lambda,\nu}_{X(\lambda,x) } \, d\mu(x) \, .
$$
Note that    for any~$(\lam,\nu)$, the map~$u^{\lambda,\nu}_{X(\lambda,x)}$
 is a group homomorphism from~$G$ into the group~$U (L^2( \mathfrak p_\lambda))$  of unitary operators
of~$L^2( \mathfrak p_\lambda)$, so functions~$f$ of~$L^1(G)$  have a Fourier transform~$\left({\mathcal F}(f)(\lambda,\nu)\right)_{\lambda,\nu}$ which is a bounded family of bounded operators on~$L^2( \mathfrak p_\lambda)$. One may check that  the Fourier transform   exchanges   convolution, whose definition is recalled in~(\ref{convolutiondef}),
and   composition:
\begin{equation}\label{fourconv}
 {\mathcal F}( f \star g )( \lam,\nu ) = {\mathcal F}(f) ( \lam,\nu )\circ{\mathcal F} (g)( \lam,\nu ) \, .
 \end{equation}
Besides, the Fourier transform can be extended to an isometry from~$L^2(G)$ onto the Hilbert
space of two-parameter families~$ A  = \{ A (\lam,\nu ) \}_{(\lambda,\nu) \in\mathfrak r(\Lambda)}$
 of operators on~$L^2( \mathfrak p_\lambda)$ which are
Hilbert-Schmidt for almost every~$(\lambda,\nu) \in\mathfrak r(\Lambda)$, with~$\|A (\lam,\nu )\|_{HS (L^2( \mathfrak p_\lambda))}$ measurable and with norm
\[ \|A\| := \left( \int_{\lambda\in\Lambda}\int_{\nu\in\mathfrak r^*_\lambda}
\|A (\lam,\nu )\|_{HS (L^2( \mathfrak p_\lambda))}^2 |{\mbox {Pf}} (\lambda) |d\nu \, d\lam
\right)^{\frac{1}{2}}<\infty  \, ,\]
 where~$|{\mbox {Pf}} (\lambda) | := \prod_{j=1}^d \eta_j(\lambda) \, $ is the Pfaffian of~$B(\lambda)$. 
  We have the following Fourier-Plancherel formula: there exists a constant $\kappa>0$ such that 
 \begin{equation}
\label{Plancherelformula} \int_G  |f(x)|^2  \, dx
=  \kappa \, \int_{\lambda\in\Lambda}\int_{\nu\in\mathfrak r^*_\lambda} \|{\mathcal F}(f)(\lambda,\nu)\|_{HS(L^2( \mathfrak p_\lambda))}^2 |{\mbox {Pf}} (\lambda) | \, d\nu\, d\lambda   \,.
\end{equation}
Finally, we have an inversion formula as stated in the following proposition which is proved in the Appendix page~\pageref{appendixinversion}.

\begin{proposition}\label{inversioninS} There exists $\kappa>0$ such that for~$ f \in {\mathcal S}(G)$ and for almost all~$x \in G$ the following inversion formula holds:
\begin{equation}
\label{inversionformula} f(x)
= \kappa \, \int_{\lambda\in\Lambda}\int_{\nu\in\mathfrak r^*_\lambda} {\rm{tr}} \, \Big((u^{\lambda,\nu}_{X(\lambda,x)})^\star {\mathcal F}f(\lambda,\nu)  \Big)\, |{\mbox {Pf}} (\lambda) |\, d\nu\,d\lambda \,.\end{equation}
 \end{proposition}

\subsubsection{The sublaplacian} \label{freq}
 Let~$(V_1,\dots,V_{m})$ be an orthonormal basis of ${\mathfrak g}_1$,  then the sublaplacian on $G$ is defined by
\begin{equation}
\label{DeltaG}
  \Delta_{G}:= \sum_{j = 1}^{m } V_j^2.
 \end{equation}
It is a self-adjoint operator which is independent of the orthonormal basis ~$(V_1,\dots,V_{m})$, and  homogeneous of degree $2$ with respect to the dilations in the sense that : 
 $$
 \delta_{t}^{-1}\Delta_G \, \delta_t = t^2 \Delta_G \, .
 $$
To write its expression in Fourier space, we consider the basis of Hermite functions $(h_n)_{n\in\N}$,  normalized in~$L^2(\R) $ and satisfying
for all real numbers~$\xi $:
$$
h''_n(\xi)-\xi^2 h_n(\xi)= -(2n+1) h_n(\xi) \, .
$$  
Then, for any multi-index~$\alpha \in {\mathbb N}^d$,  we define the functions $h_{\alpha,\eta(\lambda)}$ by
\begin{equation}\label{defhn}
\begin{aligned}
 \forall \, \Xi = (\xi_1,\dots, \xi_d) \in \R^d \, , \quad h_{\alpha,\eta(\lambda)} (\Xi) & :=\prod_{j=1}^d h_{\alpha_j,\eta_j(\lambda)}(\xi_j) \quad \mbox{and} \\
 \forall (n,\beta) \in {\mathbb N}\times\R^+ \, , \forall \xi \in \R \, , \quad 
 h_{n,\beta } (\xi) &  :=  \beta^{\frac 1 4}  h_{n} \big( \beta^{\frac 1 2} \xi \big) \, .
\end{aligned}
\end{equation}
The sublaplacian~$\Delta_G$ defined in~(\ref{DeltaG})  satisfies
\begin{equation}\label{formulafourierdelta}
{\mathcal F}(- \Delta_G f) (\lambda,\nu) = {\mathcal F}(  f) (\lambda,\nu)  \left(H(\lambda)+|\nu|^2\right)\, , 
\end{equation}
where $|\nu|$ denotes the euclidean norm of the vector $\nu$ in $\R^k$ and~$ H(\lambda)$ is the  diagonal  operator  defined on 
$L^2(\R^d)$ by
$$
H(\lambda) h_{\alpha,\eta(\lambda)} =   \sum_{j =1}^d (2 \alpha_j + 1) \eta_j(\lambda) \,  h_{\alpha,\eta(\lambda)}\, .
$$
In  the following we shall denote the ``frequencies" associated with~$P_j^2(\lambda) + Q_j^2(\lambda) $ by
\begin{equation}\label{eq:freqxy}
\zeta_j (\alpha, \lambda) :=   (2 \alpha_j + 1) \eta_j(\lambda)  \, , \quad (\alpha, \lambda) \in {\mathbb N}^d \times \Lambda \, ,
\end{equation}
and those associated with~$H(\lambda)$ by 
\begin{equation}\label{eq:freq}
\zeta (\alpha, \lambda) :=  \sum_{j =1}^d \zeta_j (\alpha, \lambda)   \, , \quad (\alpha, \lambda) \in {\mathbb N}^d \times \Lambda \, .
\end{equation}
 Note that~$\Delta_G $ is directly related to the harmonic oscillator    via~$H(\lambda)$ since   eigenfunctions associated with the eigenvalues~$\zeta (\alpha, \lambda) $   are the products of 1-dimensional Hermite functions.  Also observe that~$\zeta (\alpha, \lambda)$ is smooth and   homogeneous   of degree one in~$\lambda = (\lambda_1,\dots, \lambda_p)$.  Moreover, $\zeta(\alpha,\lambda)=0$ if and only if $B(\lambda)=0$, or equivalently by~(\ref{skw}), $\lambda=0$.  %
\medskip

\noindent Notice  also that there is a difference in homogeneity in the variables $\lambda$ and $\nu$. Namely, in the variable~$\nu$, the sublaplacian acts as in the euclidean case (homogeneity $2$) while in $\lambda$, it has the homogeneity $1$ of a wave operator.

\medskip

\noindent
Finally, for any smooth function $\Phi$, we define
the operator $\Phi\left(-\Delta_{G}\right)$  by the formula
\begin{equation}\label{def}   {\mathcal  F} \big(\Phi(- \Delta_G) f\big)(\lam,\nu):=  \Phi(H(\lambda)+|\nu|^2) {\mathcal  F}  ( f )(\lam,\nu) \, ,\end{equation}
which also reads
$$
  {\mathcal  F} \big (\Phi(- \Delta_G) f\big)(\lam,\nu)  h_{\alpha,\eta(\lambda)} := 
  \Phi\left(|\nu|^2+\zeta(\alpha,\lambda)\right) {\mathcal F}(f)(\lam,\nu)
h_{\alpha,\eta(\lambda)} \,,
$$  for all $(\lambda,\nu)\in \mathfrak r(\Lambda)$  and $\alpha\in{\mathbb N}^d.$ \\

\subsubsection{Strict spectral localization} Let us introduce  the following notion of spectral localization, which we shall call strict  spectral  localization  and which will be very useful in the following. 
 \begin{definition}\label{def:strispecloc}
A  function $f$ belonging to  $L^1(G)$ is said to be strictly spectrally localized in a set~${\mathcal C}\subset \R$ if there exists a smooth function $\theta$,  compactly supported in ${\mathcal C}$, such that for all~$ 1\leq  j\leq d$,  
\begin{equation}\label{strictloceq}
{\mathcal F}(f)(\lambda,\nu)={\mathcal F}(f)(\lambda,\nu) \, \theta \big ((P_j^2+Q_j^2)(\lambda )    \big) \, ,\;\;\forall (\lambda,\nu)\in \mathfrak r(\Lambda) \, .
\end{equation}
\end{definition}
\begin{remark} {\rm    One could expect the notion of spectral localization to relate to the Laplacian instead of each individual vector field~$P_j^2+Q_j^2$, assuming rather the less restrictive condition 
$$
{\mathcal F}(f)(\lambda,\nu)={\mathcal F}(f)(\lambda,\nu) \, \theta \big (H(\lambda )    \big) \, ,\;\;\forall (\lambda,\nu)\in \mathfrak r(\Lambda) \, .
$$
The choice we make here is more restrictive due to the anisotropic context (namely the fact that~$\eta_j(\lambda)$ depends on~$j$):   in the case of the Heisenberg group or more generally  H-type groups, the notion of  ``strict spectral localization'' in a ring~${\mathcal C}$ of~$\R^p$ actually coincides with the more usual definition of ``spectral localization" since as recalled in the next paragraph~$\eta_j(\lambda) = 4 |\lambda|$ (for a
complete presentation and more details on spectrally localized functions, we refer the reader to~\cite{bg, bfg, bfg2}).  Assumption~(\ref{strictloceq})
 guarantees 
a lower bound, which roughly states that for~${\mathcal F}(f)(\lambda,\nu)  h_{\alpha,\lambda}$ to be non zero, then
\begin{equation}\label{lowerbound}
\forall j\in\{1,\dots,d\},\;\;
(2 \alpha_j + 1) \eta_j(\lambda) \geq c >0 \, , 
\end{equation}
hence each~$\eta_j$ must be bounded away from zero, rather than  the sum over~$j$. These lower bounds are important ingredients of the proof (see Section~\ref{prooflemmas}).  }
\end{remark}

\subsection{Examples}

   \medskip

   Let us give a few examples of well-known stratified Lie groups with a step 2 stratification. 
Note that   nilpotent  Lie groups which are connected, simply
connected and whose Lie algebra  admits a step 2 stratification are called {Carnot groups}.

\subsubsection{The Heisenberg group} The Heisenberg group $\H^d$  is defined as the space $\R^{2d+1}$ whose elements  can be written $w =
(x,y,s) $  with $(x,y) \in
   \R^{d} \times \R^{d} $, endowed with the following product law:
   \[
(x,y,s)\cdot (x',y',s') = (x+x',y+y',s+s'- 2(x\mid y') + 2(y\mid
x')),
\]
where~$(\cdot\mid \cdot)$ denotes the euclidean scalar product on~$\R^d$.  In that case the center consists in the points of the form $(0,0,s)$ and is of dimension 1.  The Lie algebra of left invariant vector fields  is
generated by
 $$
 X_j\! :=\!\partial_{x_j} + 2 y_j \partial_s \, ,\!\!\!\quad Y_{j} \!   :=\!
\partial_{y_j} - 2 x_j \partial_s\!\!\!\quad\hbox{with}\!\!\!\quad 1\leq j\leq d \, , \!\!\! \quad 
 \hbox{and}\!\!\!\quad
 S   := \partial_s=\frac 1 4[Y_j,X_j ] \, .
 $$
 The canonical skew-symmetric form~$B(\lambda)(U,V)$ defined in~(\ref{skw}) associated with the frequencies $\lambda\in\R^*$  is proportional to~$\lambda$, since~$[U,V]$ is proportional to~$\partial_s$.
 Its radical reduces to $\{0\}$ with $\Lambda=\R^*$ and~$|\eta_j(\lambda)|=4 \, |\lambda|$ for all $j\in\{1,\dots,d\}$. Note in particular that  strict spectral localization and spectral localization are equivalent.

\subsubsection{ H-type groups} These groups are  canonically isomorphic to $\mathbb R^{m+p}$, and are a multidimensional version of the Heisenberg group. The group law is of the form
\begin{equation*}
\quad \quad \quad (x^{(1)},x^{(2)}) \cdot (y^{(1)},y^{(2)}):=\begin{pmatrix}
x_j^{(1)}+y_j^{(1)},\,\,\,j=1,...,m \\x_j^{(2)}+y_j^{(2)}+\frac12 \langle x^{(1)}, U^{(j)}y^{(1)} \rangle,\,\,\,j=1,...,p
\end{pmatrix}
\end{equation*}
where $U^{(j)}$ are $m \times m$ linearly independent orthogonal skew-symmetric matrices satisfying the property
$$U^{(r)}U^{(s)}+U^{(s)}U^{(r)}=0$$
for every $r,s \in \left \{1,...,p \right \}$ with $r \neq s$.    In that case the center is of dimension~$p$ and may be identified with~$\R^p$ and the radical of the canonical skew-symmetric form associated with the frequencies $\lambda$ is again $\{0\}$.
 For example the Iwasawa subgroup of semi-simple Lie groups of split rank one (see~\cite{koranyi2}) is of this type. On  H-type groups, $m$ is an even number which we   denote by $2\ell$ and the Lie algebra of left invariant vector fields  
is spanned by the following vector fields, where we have written~$ z=(x,y) $ in~$ \R^{\ell} \times \R^{\ell}$: for~$
j$ running from~1 to~$\ell$ and~$k$ from~$ 1$ to~$p$,
 \[ 
 X_j\! :=\!\partial_{x_j} +  \frac 1 2 \sum^{p}_{k=1}\sum^{2\ell}_{l=1}z_l \, U_{l,j}^{(k)}\partial_{s_k} \, ,\!\!\!\quad Y_{j} \!:=\!
\partial_{y_j}  +\frac 1 2 \sum^{p}_{k=1}\sum^{2\ell}_{l=1}z_l \,U_{l,j+\ell}^{(k)}\partial_{s_k}
  \quad  \hbox{and}\, \quad \partial_{s_k} 
 .\]
In that case,  we have $\Lambda=\R^p\setminus\{0\}$ with  $\eta_j(\lambda)=\sqrt{\lambda^2_1+\cdots+\lambda^2_p}$ for all  $j \in \{1,\dots,\ell\}$ (here again, strict spectral localization and spectral localization are equivalent).
 
\subsubsection{Diamond groups} These groups  which occur in    crystal theory (for more details, consult \cite{Ludwig, Poguntke}),  are of the type~$\Sigma \ltimes \H^d $ where~$\Sigma$ is a connected Lie group acting smoothly on~$ \H^d $.    One can find examples for which the  radical of the canonical skew-symmetric is of any dimension $k$, $0\leq k\leq d$. For example, one can take for $\Sigma$ the~$k$-dimensional torus, acting on $\H^d$ by 
$$
\theta( w):=(\theta\cdot z,s):=\left({\rm e}^{i\theta_1}z_1,\dots,{\rm e}^{i\theta_k}z_k,z_{k+1},\dots,z_d,s\right),\;\;w=(z,s)
$$
where the element $\theta=(\theta_1,\dots,\theta_k )$ corresponds to the element $\left({\rm e}^{i\theta_1},\dots,{\rm e}^{i\theta_k}\right)$ of~${\mathbb T}^k$.
Then, the product law on $G={\mathbb T}^k\ltimes\H^d$ is given by 
$$(\theta,w)\cdot (\theta',w')=\big(\theta+\theta',w.(\theta( w')\big )\, ,$$
where $w.(\theta( w'))$ denotes the Heisenberg product of $w$ by $\theta( w')$. As a consequence, the center of $G$ is of dimension $1$ since it consists of the points of the form $(0,0,s)$ for $s\in\R$. 
Let us choose for simplicity $k=d=1$, the algebra of left-invariant vector fields is generated by  the vector fields $\partial_\theta$, $\partial_s$, $\Gamma_{\theta,x}$ and $\Gamma_{\theta,y}$
where
\begin{eqnarray*}
\Gamma_{\theta,x} &= &{\rm cos}\, \theta \partial_x +{\rm sin}\, \theta \partial_y +2(y{\rm cos}\,\theta-x{\rm sin} \theta)\partial_s,\\
\Gamma_{\theta,y} &= &-{\rm sin}\, \theta \partial_x +{\rm cos}\, \theta \partial_y -2(y{\rm sin}\,\theta +x{\rm cos }\theta)\partial_s.
\end{eqnarray*}
%
%
%
%, for~$1\leq j\leq k $ and~$1\leq \ell\leq d$: 
% \[ T_j=(\partial_{\theta_j},0)\, , \quad \widetilde S= (0,S)  \, , \quad \widetilde X_\ell=(0,X_\ell) \;\;{\rm and}\;\;\widetilde Y_\ell=(0,Y_\ell) \,,\]
%and we have 
%$$[T_j,\widetilde X_\ell]=[T_j,\widetilde Y_\ell]=0\, ,\;\; \frac 1 4 [\widetilde Y_\ell,\widetilde X_\ell]=(0,S)\, .$$
It is not difficult to check that the the radical of $B(\lambda)$
 is  of dimension~$1$. In the general case,  where $k\leq d$, the algebra of left-invariant vector fields is generated by  the vector fields $\partial_s$, the $2 (d-k)$ vectors
 $$X_{\ell}=\partial_{x_\ell}+2y_{\ell} \partial_s,\;\;Y_{\ell}=\partial_{y_\ell}-2x_\ell\partial_s,$$
 and  the $3k$ vectors defined for $1\leq j\leq k$ by : $\partial_{\theta_j}$,  $\Gamma_{\theta_j,x_j}$ and $\Gamma_{\theta_j,y_j}$
where
\begin{eqnarray*}
\Gamma_{\theta_j,x_j} &= &{\rm cos}\, \theta_j \partial_{x_j} +{\rm sin}\, \theta_j \partial_{y_j} +2(y_j{\rm cos}\,\theta_j-x_j{\rm sin} \theta_j)\partial_s,\\
\Gamma_{\theta_j,y_j} &= &-{\rm sin}\, \theta_j \partial_{x_j} +{\rm cos}\, \theta_j \partial_{y_j} -2(y_j{\rm sin}\,\theta_j +x_j{\rm cos }\theta_j)\partial_s.
\end{eqnarray*}
and this provides an example with a radical of dimension~$k$.
 
 \subsubsection{The tensor product of Heisenberg groups} Consider  $\H^{d_1} \otimes \H^{d_2}$  the set of elements~$(w_1,w_2) $ in~$\H^{d_1} \otimes \H^{d_2}$, which can be written~$(w_1,w_2)= (x_1,y_1,s_1,x_2,y_2,s_2)$ in $\R^{2d_1+1} \times \R^{2d_2+1}$, equipped with the law of product: 
   \[
(w_1,w_2)\cdot (w_1',w_2') = (w_1\cdot w_1',w_2 \cdot w_2'),
\] where $w_1\cdot w_1'$ and $w_2 \cdot w_2'$ denote respectively the product in $\H^{d_1}$ and $ \H^{d_2}$.  Clearly  $\H^{d_1} \otimes \H^{d_2}$ is a step~2 stratified  Lie group with center  of dimension $2$  and   radical index null. Moreover, for $\lambda=(\lambda_1,\lambda_2)$ in the dual of the center, the canonical skew bilinear form $B(\lambda)$ has radical $\{0\}$ with $\Lambda=\R^*\times\R^*$,  and one has~$\eta_1(\lambda)= 4 \, |\lambda_1|$ and $\eta_2(\lambda)= 4 \, |\lambda_2|$. In that case, strict spectral localization is a more restrictive condition than spectral localization. Indeed, if $f$ is spectrally localized, one has $\lambda_1\not=0$ {\bf or} $\lambda_2 \neq 0$ on the support of ${\mathcal F}(f)(\lambda)$, while one has  $\lambda_1\not=0$ {\bf and} $\lambda_2 \neq 0$ on the support of ${\mathcal F}(f)(\lambda)$ if $f$ is strictly spectrally localized.

\subsubsection{ Tensor product of  H-type groups} The group  $\mathbb R^{m_1+p_1} \otimes \mathbb R^{m_2+p_2}$ is easily verified to be a step~2 stratified  Lie group with center  of dimension $p_1+p_2$,    a radical index null and a skew bilinear form~$B(\lambda)$ defined on $\R^{m_1+m_2}$ with $m_1=2\ell_1$ and $m_2=2\ell_2$. The Zariski open set associated with $B$ is  given by~$\Lambda=(\R^{p_1}\setminus\{0\})\times (\R^{p_2}\setminus\{0\})$ and  for $\lambda=(\lambda_1,\cdots,\lambda_{p_1+p_2} )$, we have
\begin{equation} \label{H1}
\begin{aligned}
\eta_j(\lambda)  &= \sqrt{\lambda^2_1+\dots+\lambda^2_{p_1} },  \quad \mbox{for all}   \quad j \in \{1,\dots,\ell_1\} \quad \mbox{and} \\
\eta_j (\lambda) &= \sqrt{\lambda^2_{p_1+1}+\dots+\lambda^2_{p_1+p_2}}  \quad \mbox{for all}  \quad j \in \{\ell_1+1,\dots,\ell_1+\ell_2\}. 
\end{aligned}
\end{equation}

 \subsection{Main results}

The purpose of this paper is to establish optimal  dispersive inequalities for the linear Schr\"odinger  equation  on {step 2 stratified}  Lie groups  associated  with the sublaplacian. In view   of~(\ref{formulafourierdelta}) and the fact that the ``frequencies"~$ \zeta (\alpha,\lambda)$ associated with~$H(\lambda)$ given by 
\eqref{eq:freq} are homogeneous   of degree one in~$\lambda$, the Schr\"odinger  operator on~$G$ behaves 
 like a wave operator   on a space of the same dimension~$p$ as the center of  $G$, and like a Schr\"odinger operator on a space of the same dimension~$k$ as the radical of the canonical skew-symmetric form. By comparison with the classical dispersive estimates, the expected result would be a dispersion phenomenon with an optimal rate of decay of order $ | t   |^{- \frac{k+p-1} 2}$. However, as   will be seen through   various examples, this  anticipated  rate is not always achieved. To reach this maximum rate of dispersion, we   require a condition on~$ \zeta (\alpha,\lambda)$. 
  \medskip

\begin{assumption}\label{keyp} 
For each multi-index~$\alpha $ in~$ {\mathbb N}^d$, the Hessian matrix of the map~$\lambda \mapsto \zeta (\alpha,\lambda)$ satisfies
$$
{\rm rank } \,D^2_\lambda \zeta (\alpha,\lambda) = p-1
$$
where~$p$ is the dimension of the center of~$G$.
\end{assumption}

\begin{remark} {\rm  As was observed  in Paragraph {\rm\ref{freq}},   $\zeta (\alpha,\lambda)$ is a smooth function, homogeneous   of degree one on  $\Lambda$. By homogeneity arguments, one therefore has~$D_\lambda^2 \zeta (\alpha,\lambda) \lambda = 0$. It follows that there always holds
$$
{\rm rank } \,D^2_\lambda \zeta (\alpha,\lambda) \leq p-1 \, , 
$$
hence Assumption~\ref{keyp} may be understood as a maximal rank property.
}
\end{remark}

 \noindent Let us now present the dispersive inequality for the  Schr\"odinger equation. Recall that the  linear Schr\"odinger  equation  writes as follows on $G$:
\begin{equation}\label{eq:sh}
\left\{\begin{array}{l}
\left(i\partial_t -  \Delta_{G}\right) f=0\\
f_{|t=0}=f_0\, ,
\end{array}\right.
\end{equation}
where the function $f$ with complex values depends on $(t,x) \in \R \times G$.

\begin{theorem}\label{dispgrad}  Let $G$ be a {step {\rm2} stratified}  Lie group    with center of dimension $p$ with $1\leq p < n$ and   radical index $k$. Assume that  Assumption~{\rm\ref{keyp}} holds.  A  constant $C$ exists such that if   $f_0$ belongs to~$L^1(G)$ and is strictly spectrally localized in a ring of~$\R$ in the sense of Definition~{\rm\ref{def:strispecloc}}, then   the associate solution~$f$ to the Schr\"odinger  equation \refeq{eq:sh} satisfies
\begin{equation}\label{eq:gradeddispS}
\|f(t, \cdot )\|_{ L^\infty(G)}  \leq \frac {C} {| t   |^{ \frac{k}{2}}(1+  | t   |^{\frac {p- 1} {2 }} )} \|f_0\|_{ L^1(G)} \, , 
\end{equation}
for all $t\neq 0$ and the result is sharp in time.
\end{theorem}

\noindent The fact that a spectral localization is required in order to obtain the dispersive estimates is not surprising. Indeed, recall that in the $\R^d$ case for instance, the dispersive estimate for the Schr\"odinger equation  derives immediately (without any spectral localization assumption) from the fact that the solution $u(t)$ to the free Schr\"odinger equation  on $\R^d$ with Cauchy data $u_0$ writes  for~$t \neq 0$
$$ u(t, \cdot ) = u_0 *\frac 1 {(-2i \pi t)^{\frac {d}2} } {\rm e}^{-i \frac {|\cdot|^2}{4 t}} \,,$$ 
where $*$ denotes the convolution product in $\R^d$ (for a detailed proof of this fact, see for instance  Proposition 8.3  in \cite{bcd}). However proving   dispersive estimates for the wave equation in $\R^d$ requires more elaborate techniques (including oscillating  integrals) which involve an assumption of spectral localization  in a ring. 
In the case of a {step 2 stratified}  Lie group $G$, the main difficulty arises from the complexity of the expression of Schr\"odinger  propagator that mixes  a wave  operator  behavior with that of a Schr\"odinger operator. This explains on the one hand the decay rate in Estimate \eqref{eq:gradeddispS} and on the other hand the hypothesis of strict  spectral localization.

   \medskip

\noindent Let us now discuss  Assumption~\ref{keyp}. 
As mentioned above, there is  no dispersion   phenomenon for the Schr\"odinger  equation on the Heisenberg group $\H^d$ (see~\cite{bgx}).  Actually the same holds  for the tensor  product of Heisenberg  groups $\H^{d_1} \otimes \H^{d_2}$ whose center is of dimension $p=2$ and   radical index   null, and more generally to the case   of  2 step  stratified Lie groups,  decomposable on non trivial 2 step  stratified Lie groups :  we derive indeed from Theorem~\ref{dispgrad} the following corollary. 
\begin{corollary}\label{cor}
Let $G=\otimes_{1\leq m\leq r} G_m$ be a decomposable, $2$ step stratified Lie group 
where the groups~$G_m$ are non trivial $2$-step  stratified  Lie groups satisfying Assumption~{\rm\ref{keyp}}, of radical index~$k_m$ and with centers of dimension $p_m$. Then the dispersive estimates holds with rate~$|t|^{-q}$,
$$q:= {1\over 2}\sum_{1\leq m\leq r} \left(k_m+p_m-1\right)={1\over 2} (k+p-r) \, ,$$
where $p$ is the dimension of the center of $G$ and $k$ its radical index. Besides, 
this rate is optimal.
\end{corollary}
 \noindent Corollary~\ref{cor} is a direct consequence of Theorem~\ref{dispgrad} and  the simple observation that $\Delta_G= \otimes_{1\leq m\leq r}  \Delta_{G_m}.$
This result  applies for example  to the 
tensor product of Heisenberg groups, for which there is no dispersion,  and  to the tensor product  of   H-type  groups $\mathbb R^{m_1+p_1} \otimes \mathbb R^{m_2+p_2}$ for which the dispersion rate is $ t^{- (p_1+p_2-2)/2}$ (see~\cite{hiero}). 
Corollary~\ref{cor} therefore shows that  it can happen that the ``best'' rate of decay~$ | t   |^{-(k+p-1)/2}$ is not reached, in particular for  decomposable Lie groups. This suggests that  Assumption  \ref{keyp} could be related with decomposability.

 \medskip 
 \noindent
  More generally, a large class of groups  which do not satisfy the Assumption~\ref{keyp} is given by  step~2 stratified  Lie group~$G$  for which~$\zeta(0, \lambda)$ is a  linear form  on    each connected component of the Zariski-open subset $\Lambda$.  Of course, the Heisenberg group and any tensor product of Heisenberg group is of that type.  We then have the following result which illustrates that there exists examples of groups without any dispersion and which do not satisfy Assumption~\ref{keyp}.

\begin{proposition}\label{remnodisp}
Consider a  step~$2$ stratified  Lie group~$G$ whose radical  index is null and for which~$\zeta(0, \lambda)$ is a  linear form  on    each connected component of the Zariski-open subset $\Lambda$. Then, there exists $f_0\in{\mathcal S}(G)$, $x\in G$ and $c_0>0$ such that 
$$\forall t\in\R^+,\;\;|{\rm e}^{-it\Delta_G}f_0(x)|\geq c_0.$$
\end{proposition}

\medskip 
\noindent Finally we point out that the dispersive estimate given in Theorem~\ref{dispgrad} can be regarded as a first step towards space-time estimates of the Strichartz type. However due to the strict spectral localization assumption, the   Besov spaces which should appear in the study (after summation over frequency bands) are naturally anisotropic;   thus proving such estimates is likely to be very technical, and is postponed to future works.

\subsection{Strategy of the proof of Theorem~\ref{dispgrad}}
 In  the  statement of  Theorem~\ref{dispgrad}, there are two different results: the dispersive estimate in itself on the one hand, and its optimality on the other. 
 Our strategy or proof is closely related to the method developed in~\cite{bgx} and~\cite{hiero} with additional non negligible technicalities.
 
 \medskip
  \noindent   
In the situation of~\cite{bgx}  where  the  Heisenberg group $\H^d $ is considered, the authors  prove that there is no dispersion by
exhibiting explicitly a Cauchy data~$f_0$ for which the solution~$f(t,\cdot)$ to the Schr\"odinger equation~(\ref{eq:sh}) satisfies
\begin{equation}\label{eq:nd} \forall \,q \in [1,\infty] \, , \quad    \|f(t,\cdot)\|_{ L^q(\H^d)}  =  \|f_0\|_{ L^q(\H^d)} \, .
\end{equation}
More precisely,
they take advantage of the fact that the Laplacian-Kohn operator~$ \Delta_{ \H^d  }$ can be recast under the form
\begin{equation}
\label{eq:hk} \Delta_{ \H^d  }= 4\sum_{j=1}^{d}  (Z_j \overline Z_j +i\partial_s )\,,
\end{equation}
where $\bigl\{Z_1, \overline Z_1,  ..., \ Z_d, \overline Z_d, \partial_s \bigr\}$ is the canonical basis of Lie algebra of left invariant vector fields on~$\H^d$ (see~\cite{bfg} and the references therein for  more details). This implies that for a non zero function $f_0$ belonging to~$ \mbox  {Ker} \: \big(\sum_{j=1}^d Z_j \overline Z_j \big)$, the solution of the Schr\"odinger equation on the Heisenberg group~$f(t)={\rm e}^{-it\Delta_{\H^d}}f_0$ actually solves a transport equation:
$$f(z,s,t)= {\rm e}^{4d t\partial_s}f_0(z,s)=f_0 (z, s + 4dt) $$
and hence
satisfies  \refeq{eq:nd}.  
The arguments used in \cite{hiero} for general~H-type groups are similar to the ones developed in~\cite{bgx}: the dispersive estimate is obtained using an explicit formula for the solution, coming from Fourier analysis,  combined with a stationary phase theorem. The Cauchy data used to prove the optimality is again in the kernel of an adequate operator, by a decomposition similar to~(\ref{eq:hk}). 

\medskip

 \noindent   As in  \cite{bgx}   and \cite{hiero}, the first step of the proof of Theorem~\ref{dispgrad}  consists in writing an explicit formula for the solution of the equation by use of  the Fourier transform.   Let us point out that    in the setting of~\cite{bgx}   and~\cite{hiero},    irreducible representations are isotropic with respect to the  dual of the center of the group; this isotropy allows to  reduce to a one-dimensional framework and deduce the dispersive effect from a   careful use of a stationary phase argument of~\cite{stein3}. As we have already seen in Paragraph \ref{defirreducible}, the irreducible representations are no longer isotropic in the general case of stratified Lie groups, and thus we    adopt a more technical  approach making use of Schr\"odinger  representation and taking advantage of some properties of  Hermite functions appearing in the  explicit representation of the solutions derived by Fourier analysis (see Section~\ref{prooflemmas}).  The optimality of the inequality is   obtained  as in~\cite{bgx}   and~\cite{hiero}, by an adequate  choice of the initial data.

 \subsection{Organization of the paper}

 The article is organized as follows.    In Section~\ref{sec:explicit}, we write an explicit formulation of the solutions of the Schr\"odinger equation. Then, Section~\ref{sec:di}  is devoted to
 the proof of Theorem~\ref{dispgrad} and in
 Section~\ref{optimality},  we discuss the  optimality of the result and prove Proposition~\ref{remnodisp}.

 \medskip
 
 \noindent Finally, we mention that the letter~$C$ will be used to denote a universal constant
which may vary from line to line. We also use~$A\lesssim B$  to
denote an estimate of the form~$A\leq C B$   for some
constant~$C$.

 \medskip
 
\noindent {\bf Acknowledgements. } The authors wish to thank Corinne Blondel, Jean-Yves Charbonnel, Laurent Mazet, Fulvio  Ricci and Mich\`ele Vergne  for enlightening discussions. They also extend their gratitude to the anonymous referee for numerous remarks which improved the presentation of this paper, and for providing a   simpler and more conceptual proof of Lemma~3.6 than our original proof.

%%%%%%%%%%%%%%%%%%%%%%%%%%%%%%%%%%%%%
\section{Explicit representation of the solutions}\label{sec:explicit}

 \subsection{The convolution kernel}\label{stationaryphase*} 
Let~$f_0 $ belong to~$ {\mathcal S}(G)$ and let us consider~$f(t,\cdot )$ the solution to the free Schr\"odinger equation~(\ref{eq:sh}). In view of  \eqref{formulafourierdelta}, we have 
$$
{\mathcal F}(f(t,\cdot)) (\lambda,\nu) =  
{\mathcal F}(f_0) (\lambda,\nu)\,  {\rm e}^{it|\nu|^2+it H(\lambda)}\, ,
$$
which implies easily (arguing as in the Appendix) that $f(t,\cdot)$ belongs to $ {\mathcal S}(G)$. Assuming that $f_0$ is  strictly spectrally localized in the sense of Definition \ref{def:strispecloc},    there exists a smooth function $\theta$  compactly supported in a ring ${\mathcal C}$ of~$\R$ such that, defining
$$
\Theta (\lambda) :=  \prod_{j=1}^d  \theta \big ((P_j^2 + Q_j^2)(\lambda )    \big) \, , 
$$
then
$$
{\mathcal F}(f(t,\cdot)) (\lambda,\nu) = 
 {\mathcal F}(f_0) (\lambda,\nu) \, \Theta (\lambda) \,  {\rm e}^{it|\nu|^2+it H(\lambda)}\, .
$$ 
\noindent Therefore by  the inverse Fourier transform~(\ref{inversionformula}),
we deduce that the function~$f (t,\cdot)$   may be decomposed in the following way: 
\begin{equation}\label{formulaftx}
 f(t,x) = \kappa \, \int_{\lambda\in\Lambda}\int_{\nu\in\mathfrak r^*_\lambda}  {\rm{tr}} \, \Big((u^{\lambda,\nu}_{X(\lambda,x)})^\star \,  {\mathcal F}(f_0) (\lambda,\nu) \,  \Theta (\lambda) \,  {\rm e}^{it|\nu|^2+it H(\lambda)} \Big)   |{\mbox {Pf}} (\lambda) |  \,  d\nu\,d\lambda\, .
\end{equation}
We set for $ X\in\R^{n}$, 
\begin{equation}\label{defkt}
k_t(X) := \kappa \,  \int_{\lambda\in\Lambda}\int_{\nu\in\mathfrak r^*_\lambda} {\rm{tr}} \, \left(u^{\lambda,\nu}_X \,  \Theta (\lambda) \,  {\rm e}^{it|\nu|^2+it H(\lambda)} \right)  |{\mbox {Pf}} (\lambda) |  \,d\nu d\lambda \,.
\end{equation}
The function $k_t$ plays the role of a convolution kernel in the variables of the Lie algebra 
and we have the following result.

\begin{proposition}\label{firstreduction}
If the function~$k_t$ defined in~{\rm(\ref{defkt})} satisfies
\begin{equation}\label{Linftyboundk}
\forall t\in  \R \, , \quad \|  k_t \|_{L^\infty(\R^{n})} \leq \frac {C} {| t   |^{ \frac{k}{2}}(1+  | t   |^{\frac {p- 1} {2 }} )}\, \virgp
\end{equation}
then Theorem~{\rm\ref{dispgrad}} holds.
\end{proposition}

\begin{proof}
We write, according to~(\ref{formulaftx}),
\begin{eqnarray*}
 f(t,x) &  =  & \kappa \, \int_{\lambda\in\Lambda}\int_{\nu\in\mathfrak r^*_\lambda} \int_{y\in G} {\rm{tr}} \, \Big((u^{\lambda,\nu}_{X(\lambda,x)} )^* u^{\lambda,\nu}_{X(\lambda,y)} \,  \Theta (\lambda) \,  {\rm e}^{it|\nu|^2+it H(\lambda)} \Big) f_0(y)  |{\mbox {Pf}} (\lambda) |  \,  d\nu\,d\lambda\,d\mu(y)\, \\
 & = & 
  \kappa \, \int_{\lambda\in\Lambda}\int_{\nu\in\mathfrak r^*_\lambda} \int_{y\in G} {\rm{tr}} \, \Big( u^{\lambda,\nu}_{X(\lambda,y)} \,  \Theta (\lambda) \,  {\rm e}^{it|\nu|^2+it H(\lambda)} \Big) f_0(x \cdot y)  |{\mbox {Pf}} (\lambda) |  \,  d\nu\,d\lambda\,d\mu(y)\, .
\end{eqnarray*}
Note that we have used the property that the map~$X \mapsto u^{\lambda,\nu}_X$ is a unitary representation, and the invariance of the Haar measure by translations. 

\smallskip \noindent 
Now we use the  exponential law $y\mapsto Y=(P(\lambda,y), Q(\lambda,y),Z, R(\lambda,y))$ and the fact that $d\mu(y)=dY$ the Lebesgue measure, then we perform a linear orthonormal change of  variables $$(P(\lambda,y),Q(\lambda,y),R(\lambda,y))\mapsto (\tilde P,\tilde Q,\tilde R)$$ so that $d\mu(y)=dY=d\tilde P\,d\tilde Q \,dZ\, d\tilde R$ and  we write
$$\displaylines{\qquad
 f(t,x) = \kappa \, \int_{\lambda\in\Lambda}\int_{\nu\in\mathfrak r^*_\lambda} \int _{(\tilde P,\tilde Q,Z,\tilde R)\in \R^{n}} {\rm{tr}} \, \Big(u^{\lambda,\nu}_{(\tilde P,\tilde Q,Z,\tilde R)} \,  \Theta (\lambda) \,  {\rm e}^{it|\nu|^2+it H(\lambda)} \Big) \hfill\cr\hfill
 \times f_0(x\cdot {\rm exp}(\tilde P,\tilde Q,Z,\tilde R))  |{\mbox {Pf}} (\lambda) |  \,  d\nu\,d\lambda\,d\tilde P\,d\tilde Q \,dZ\,d\tilde R\, .
\qquad \cr}
$$
Thanks to the Fubini Theorem and   Young inequalities, we can write (dropping the $\tilde\,$ on the variables),
\begin{eqnarray*}
|f(t,x)| & = & \left| \int _{( P, Q,Z, R)\in \R^{n}} k_t (P,Q,Z,R) f_0(x \cdot {\rm exp}( P ,Q,Z,R)) dP\,dQ\,dR\, dZ\right|\\
&\leq & \| k_t\| _{L^\infty(G)} \left| \int _{(P,Q,Z,R)\in\R^{n}} f_0(x \cdot {\rm exp}( P, Q,Z,R)) dP\,dQ\,dR\, dZ\right| \\
& \leq & \| k_t\| _{L^\infty(G)} \| f_0\|_{L^1(G)}.
\end{eqnarray*}
Proposition~\ref{firstreduction} is proved.
\end{proof}

 \noindent   In the next subsections, we make preliminary work by transforming the expression of $k_t$ and reducing the proof to   statements equivalent to~(\ref{Linftyboundk}).

%%%%%%%%%%%%%%%%%%%%%%%%%%%%%%%%%%%%%%%%%%%%
\subsection{Transformation of $k_t$: expression in terms of Hermite functions}\label{transformation}
Decomposing the operator~$H(\lambda)$ in the basis of Hermite functions, and recalling notation~(\ref{eq:freqxy}) replaces~(\ref{defkt}) with 
$$
 k_t(X)=\kappa  \sum_{\alpha \in {\mathbb N}^d} \int_{\Lambda}\int_{\nu\in\mathfrak r^*_\lambda} 
e^{it|\nu|^2+ it \zeta( \alpha,  \lambda )}  \prod_{j=1}^d  \theta  \big(\zeta_j ( \alpha,  \lambda )   \big)   \big( u^{\lambda,\nu}_{X} h_{\alpha,\eta(\lambda)} |  h_{\alpha,\eta(\lambda)} \big)   |{\mbox {Pf}} (\lambda) |  \,  d\nu\, d\lambda \, ,   \;\;X\in\R^{n} \, .
$$
Using the explicit form of $u^{\lambda,\nu}_{X}$ recalled in~(\ref{defpilambda}) we find the following result.

\begin{lemma}\label{lem22}
 There is a constant $\wt \kappa$ and a smooth function $F$ such that with the above notation, we have for $t \neq 0$ 
$$
k_t(P,Q,tZ, R)= \frac{\wt \kappa \, {\rm e}^{-i \frac {|R|^2}{ 4 t}}\,}{t^{\frac k 2} }  \, \sum_{\alpha \in {\mathbb N}^d} 
\int_{\Lambda}  {\rm e}^{it\Phi_\alpha( Z ,  \lambda )}  G_\alpha \big (P,Q,\eta(\lambda)\big)  \,  |{\rm {Pf}} (\lambda) | \, F(\lambda)\, d\lambda  \, ,
$$
where the phase $\Phi_\alpha$ is given by 
$$
\Phi_\alpha( Z ,  \lambda ):=  \zeta (\alpha,\lambda) -\lambda(Z) \, ,
$$
with Notation~{\rm(\ref{eq:freq})}, 
and the function $G_\alpha$ is given by the following formula, for all~$(P,Q,\eta)  \in \R^{3d}$: 
\begin{equation}\label{defG}
 G_\alpha(P,Q,\eta ):=  \prod_{j=1}^d    \theta  \big( (2\alpha_j+1) \eta_j \big)  \, g_{\alpha_j}\Big(\sqrt {\eta_j  } \, P_j,\sqrt {\eta_j} \, Q_j \Big)
\end{equation}
while for each~$(\xi_1,\xi_2,n)$ in~$\R^2 \times {\mathbb N}$,  using Notation~{\rm(\ref{defhn})},
\begin{equation}\label{defgxi1xi2}
g_{n} (\xi_1,\xi_2):=e^{-i\frac{\xi_1 \xi_2 }2} \int_{\R} e^{-i \xi_2 \xi} h_{n}( \xi_1 +\xi) h_{n}(\xi) \, d\xi \, . 
\end{equation}
\end{lemma}

 \noindent   Notice that~$(g_{n})_{n\in\N} $ is uniformly bounded in~$\R^2$    thanks to the Cauchy-Schwarz inequality and the fact that~$\|h_n\|_{L^2(\R) }=1$, and hence the same holds for~$(G_\alpha)_{\alpha\in\N^d}$ (in $\R^{3d}$). 

\medskip

\begin{proof} We begin by observing that for~$X = (P,Q,R,Z)$,
\begin{eqnarray*}
I& : = & \left(u^{\lambda,\nu}_X h_{\alpha,\eta(\lambda)}\,|\,h_{\alpha,\eta(\lambda)}\right)\\
 & = & {\rm e}^{-i\nu(R)-i\lambda(Z)}\int_{\R^d} {\rm e}^{-i\lambda([\xi +{1\over 2}P,Q])} h_{\alpha,\eta(\lambda)}(P+ \xi)h_{\alpha,\eta(\lambda)}(\xi)d \xi \, ,
\end{eqnarray*}
with in view of \eqref{skw}
$$\lambda\big(\big[ \xi +{1\over 2}P,Q\big]\big)=B(\lambda)\big(\xi +{1\over 2} P,Q\big)=\sum_{1\leq j\leq d} \eta_j(\lambda)Q_j\big(\xi_j+{1\over 2} P_j\big) \, .$$
As a consequence,
$$I=  {\rm e}^{-i\nu(R)-i\lambda(Z)} \prod_{1\leq j\leq d} \int_{\R} {\rm e}^{-i\eta_j(\lambda)(\xi_j+{1\over 2} P_j)Q_j} 
h_{\alpha_j,\eta_j(\lambda)}(P_j+ \xi_j)h_{\alpha_j,\eta_j(\lambda)}(\xi_j)d \xi_j \,.$$
The change of variables $\widetilde \xi_j=\sqrt{\eta_j(\lambda)}\, \xi_j$ gives, dropping the $ \, \widetilde~$ for simplicity,
$$I=  {\rm e}^{-i\nu(R)-i\lambda(Z)}\prod_{1\leq j\leq d}\int_{\R}   {\rm e}^{-i \sqrt{\eta_j(\lambda)} \,Q_j \big(\xi_j+{1\over 2}\sqrt{\eta_j(\lambda)}P_j\big)}
h_{\alpha_j}\big(\xi_j+\sqrt{\eta_j(\lambda)}P_j\big)h_{\alpha_j}(\xi_j)d\xi_j \, ,$$
which implies that
 $$ 
k_t(P,Q,tZ,R)= \kappa \, 
\sum_{\alpha \in {\mathbb N}^d} \int_{\mathfrak r( \Lambda)}
{\rm e}^{-it\lambda(Z)-i\nu( R)} {\rm e}^{it \zeta( \alpha,  \lambda )+it|\nu|^2}
 G_\alpha \big(P,Q,\eta(\lambda) \big)  |{\rm {Pf}} (\lambda) | \, d\nu\, d\lambda \, .
 $$
It is well known (see for instance Proposition 1.28  in \cite{bcd}) that for $t  \neq 0$ 
\begin{equation}\label{schrodispersion}
 \int_{\R^k} {\rm e}^{-i (\nu\mid R) +it|\nu|^2} \, d\nu = \Big(\frac{i\pi}{t }\Big)^{\frac k 2} {\rm e}^{-i \frac {|R|^2}{ 4 t}}\, \virgp
 \end{equation}
where~$(\cdot\mid \cdot)$ denotes the euclidean scalar product on $\R^k$. 
This implies that, for $t \neq 0$ 
$$
\left| k_t(P,Q,tZ, R)\right| \lesssim \frac{1}{|t|^{\frac k 2} }  \,\Bigl| \sum_{\alpha \in {\mathbb N}^d} 
\int_{\Lambda}  {\rm e}^{it\Phi_\alpha( Z ,  \lambda )}  G_\alpha \big (P,Q,\eta(\lambda)\big)  \,  |{\mbox {Pf}} (\lambda) | \, F(\lambda)\, d\lambda\Bigr|  \, ,
$$
with $F$ the Jacobian of the change of variables
$f: \mathfrak r^*_\lambda \longrightarrow \R^k\, ,$ which is a smooth function. Lemma~\ref{lem22} is proved.
\end{proof}

\subsection{Transformation of the kernel~$
k_t$: change of variable}\label{tokcov}
\bigskip
\noindent We are then reduced to establishing that the kernel $ \wt k_t(P,Q,tZ)$ defined by
$$ \wt k_t(P,Q,tZ):= \sum_{\alpha \in {\mathbb N}^d} 
\int_{\Lambda}  {\rm e}^{it\Phi_\alpha( Z ,  \lambda )}    G_\alpha\big(P,Q,\eta(\lambda)\big)  \, |{\mbox {Pf}} (\lambda) | \, F(\lambda)\, d\lambda  
$$ 
satisfies
\begin{equation}\label{simplest}
\forall t\in  \R \, , \quad \|  \wt k_t \|_{L^\infty(G)} \leq\frac {C} {1+  | t   |^{\frac {p- 1} {2} } } \, \cdotp
\end{equation}To this end, let us define~$m: = |\alpha| = \displaystyle \sum_{j=1}^d \alpha_j$, and  when~$m\neq 0$, let us  set~$\gamma := m\lambda \in \R^p$. By construction of~$\eta(\lambda)$ (which is homogeneous of degree one), we have  
\begin{equation}\label{defetatilde}
\forall m \neq 0  \, , \quad \eta(\lambda) =\widetilde \eta_m(\gamma): = \frac1m \eta(\gamma)  \,.
\end{equation}
Let us check that  if~$\lambda$ lies in the support of~$\theta \big (\zeta_j( \alpha,  \cdot )\big) $, then~$\gamma$ lies in a fixed ring~$\mathcal C$ of~$\R^p$, independent of~$\alpha$. 
 On the one hand we note that there is a constant~$C>0$ such that on the support of~$\theta \big (\zeta_j( \alpha,  \lambda )\big) $, the variable~$\gamma$ must satisfy 
\begin{equation}\label{gammabounded}
\forall m \neq 0  \, , \quad  (2 \alpha_j + 1)  \eta_j(\gamma)  \leq Cm \, ,
\end{equation}
for all $\alpha\in\N^d$ such that $|\alpha|=m$.
Since for each~$j$ we know that~$ \eta_j(\gamma)  
$ is positive and homogeneous of degree one, we infer that the function~$  \sl   \eta_j(\gamma)  
$ goes to infinity with~$|\gamma|$ so~(\ref{gammabounded})
 implies that~$\gamma$ must remain bounded on the support of~$\theta \big (\zeta_j( \alpha,  \lambda )\big) $. 
 Moreover thanks to~(\ref{gammabounded}) again, it is clear that
 the bound may be made uniform in~$m$. 

 \medskip 
 \noindent Now let us prove that~$\gamma$ may be bounded from below uniformly. We know that there is a positive constant~$c$ such that for~$\lambda$ on the support of~$\theta \big (\zeta_j( \alpha,  \lambda )\big) $, we have 
 \begin{equation}\label{constraint}
\forall m \neq 0  \, , \quad    \zeta_j (\alpha, \gamma)   \geq cm\, .
\end{equation}
Writing~$\gamma = |\gamma| \hat \gamma$ with~$\hat \gamma$ on the unit sphere of~$\R^p$,   we find 
$$
|\gamma|  \geq  \frac {cm}{\zeta_j (\alpha, \hat \gamma) ¬¨‚Ä†}  \,\cdotp $$
 Defining   
 $$
 C_j:=\max_{| \hat\gamma| = 1} \eta_j(\hat\gamma)  < \infty \, , 
 $$
it is easy to deduce that if~(\ref{constraint}) is satisfied, then
$$
   |\gamma|  \geq  \frac{cm }{(2m+d)\displaystyle \max_{1 \leq j \leq d }C_j} \, \virgp
 $$
hence~$\gamma$ lies in a fixed ring of~$\R^p$, independent of~$\alpha \neq 0$.
This fact will turn out to be important to perform the stationary phase argument.  

\medskip
\noindent Then we can rewrite the expression of~$ \wt k_t(P,Q,tZ)$ in terms of the variable~$\gamma$, which in view of the homogeneity of the Pfaffian produces the following formula:
$$
\begin{aligned}
\wt k_t(P,Q,tZ)&=   
  \int_\Lambda {\rm e}^{it\Phi_0(Z , \lambda)}  G_0(P,Q,\eta(\lambda))  |{\mbox {Pf}} (\lambda) |\,  F(\lambda)\,d \lambda\\
&\quad +   \sum_{m \in {\mathbb N^*}}\sumetage{\alpha \in {\mathbb N}^d }{|\alpha | = m} m^{-p-d}
 \int    {\rm e}^{it\Phi_\alpha(Z, \frac \gamma m)}   G_\alpha \big(P,Q, \widetilde\eta_m(\gamma) \big)   \,  |{\mbox {Pf}} (\gamma) | \, F \big(\frac \gamma m\big)\, d \gamma  \,.
\end{aligned}
$$
Note that the  series in~$m$ is convergent since the sum over~$|\alpha |= m$ contributes a power~$m^{d-1}$, whence a series of $m^{-p-1}$ which is convergent since $p \geq  1$.
Since the functions $G_\alpha$ are uniformly bounded with respect to $\alpha \in {\mathbb N}^d$  and $F$ is smooth,    there is a positive constant~$C$ such that 
$$ \forall t\in \R\, , \quad  \|  \wt k_t \|_{L^\infty(G)} \leq C\, .$$
In order to establish the dispersive estimate, it suffices then  to prove that 
\begin{equation}\label{goal} 
\forall t\neq0 \, , \quad \| \wt  k_t \|_{L^\infty(G)} \leq\frac {C} {  | t   |^{\frac {p- 1} {2}} } \, \cdotp
\end{equation}

\section{End of the proof of the dispersive estimate}
\label{sec:di}

 \noindent   In order to prove~\refeq{goal}, 
 we decompose $\wt k_t$ into two parts, writing
$$ \wt k_t(P,Q,tZ)=  k^1_t (P,Q,tZ)+ k^2_t (P,Q,tZ) \, , \quad $$ 
with, for a constant  $c_0$ to be fixed later on independently of $m$,
\begin{equation}\label{defkt1}
\begin{aligned}
&k^1_t(P,Q,tZ):= \int_{ |\nabla_\lambda \Phi_0 ( Z ,  \lambda )|\leq c_0 }  {\rm e}^{it\Phi_0(Z, \lambda)}  G_0 \big (P,Q,\eta(\lambda)\big)  |{\mbox {Pf}} (\lambda) |\, F (\lambda)\, d \lambda\\
 & + \,  \sum_{m \in {\mathbb N^*}}\sumetage{\alpha \in {\mathbb N}^d }{|\alpha | = m} m^{-p-d}
 \int_{ |\nabla_\gamma (\Phi_\alpha ( Z,  \frac \gamma m ))|\leq c_0 }    {\rm e}^{it\Phi _\alpha(Z,    \frac \gamma m)}  \, G_\alpha \big(P,Q, \widetilde\eta_m(\gamma) \big) F \big(\frac \gamma m\big) |{\mbox {Pf}} (\gamma) | \, d \gamma  \,.\end{aligned}
\end{equation}
In the following subsections, we successively show~(\ref{goal}) for $k^1_t$ and $k^2_t$.

\subsection{Stationary phase argument for $k^1_t$}\label{statphas}
To establish Estimate \eqref{goal},  let  us  first concentrate   on~$k^1_t$.  As usual  in this type of problem,  we  define   for each integral of the   series  defining~$k^1_t$ a vector field that commutes with the phase, prove an estimate for each term and finally check the convergence of the series. More precisely,  in the case when~$\alpha \neq 0$ and $t>0$ (the case $t<0$ is dealt with  exactly  in the same manner), we consider the following first order operator:
$$
{\mathcal L}_\alpha^1:=\frac{\mbox{Id} - i  \nabla_\gamma( \Phi_\alpha ( Z ,  \frac \gamma m ))\cdot \nabla_\gamma}{1+t | \nabla_\gamma( \Phi_\alpha ( Z ,  \frac \gamma m ))|^2} \, \cdot
$$
Clearly we have
$$
{\mathcal L}_\alpha^1 \, {\rm e}^{it  \Phi_\alpha ( Z ,  \frac \gamma m )}={\rm e}^{it  \Phi_\alpha ( Z ,  \frac \gamma m )} \, .$$ 

\medskip
 \noindent   
Let us admit the next lemma for the time being. 
\begin{lemma}\label{transposeofL}
For any integer~$N$, 
there is   a  smooth function~$\theta_N$  compactly supported  on a ring of~$\R^p$ and a positive constant $C_N$ such that defining 
\begin{equation}\label{def:psialpha}
\psi_{\alpha}(\gamma):=G_\alpha \big(P,Q, \widetilde\eta_m(\gamma) \big) F \big(\frac \gamma m\big)  \,  |{\rm {Pf}} (\gamma) |  \,,
\end{equation}
recalling notation~{\rm(\ref{defetatilde})}, we have  
$$
|(^t{} {\mathcal L}_\alpha^1 )^N\psi_{\alpha}(\gamma)\,| \leq  C_N \, m^N \,  \theta_N(\gamma) \,\big(1+\big|t^\frac12 \nabla_\gamma( \Phi_\alpha ( Z ,  \frac \gamma m ))\big|^2\big)^{-N} \, .
$$
\end{lemma}
\noindent Returning to~$k^1_t$, 
let us define (recalling that~$\gamma$ belongs to a fixed ring~$
{\mathcal C}$)
$$
\ds {\mathcal C}_\alpha(Z) :=\big  \{\gamma \in {\mathcal C};  | \nabla_\gamma( \Phi_\alpha ( Z ,  \frac \gamma m ))|\leq c_0\big \} 
$$
and let us write for any integer~$N$ and~$\alpha \neq 0$ (which we assume to be the case for the rest of the computations)
\begin{equation}\label{defIalpha}
\begin{aligned}
I_{\alpha} (Z)&:=  \int_{ {\mathcal C}_\alpha(Z) }    {\rm e}^{it \Phi_\alpha ( Z,  \frac \gamma m )} \psi_{\alpha}(\gamma) \, d \gamma  \\
&=   \int_{ {\mathcal C}_\alpha(Z) }    {\rm e}^{it \Phi_\alpha ( Z,  \frac \gamma m )} (^t{}  {\mathcal L}_\alpha^1  )^N \psi_{\alpha}(\gamma)  \,   d \gamma  \, , 
\end{aligned}
\end{equation}
where
 $
\psi_{\alpha}(\gamma)$ has been defined in~(\ref{def:psialpha}).
Then by Lemma~\ref{transposeofL}
we find that for each integer~$N$ there is a constant~$C_N$     such that
\begin{equation}\label{usformula}
|I_{\alpha}(Z)|\leq C_N \,m^{N } \int _{ {\mathcal C}_\alpha(Z)  }   \theta_N(\gamma) \big(1+t \big|  \nabla_\gamma( \Phi_\alpha ( Z ,  \frac \gamma m ))\big|^2\big)^{-N}\, d \gamma \, .
\end{equation}
Then  the end of the proof relies on three steps:
\begin{enumerate}
\item a careful analysis of the properties of the support of the integral,
\item a change of variables which leads to the estimate in $t^{-(p-1)/2}$,
\item a control in $m$ in order to prove the convergence of the sum over~$m
$.
\end{enumerate}
\medskip
Before entering into details for each step, let us observe that by definition, we have
 $$
 \Phi_\alpha ( Z,  \frac \gamma m ) =  \frac 1 m \big(\zeta (\alpha,\gamma) -\gamma(Z) \big)\, ,
$$ 
with $ \gamma(Z) = (A\gamma | Z) = (\gamma | ^tA Z)$ for some invertible matrix~$A$. Performing a change of variables in $\gamma$ if necessary, we can assume without loss of generality that $A={\rm Id}$.
Thus we write
\begin{equation}\label{formulaPhialphazeta}
 \nabla_\gamma( \Phi_\alpha ( Z ,  \frac \gamma m ))= \frac 1 m \big(\nabla_\gamma \zeta (\alpha,\gamma) - Z \big)\, .
\end{equation}

\subsubsection{Analysis of the support of the integral defining $I_\alpha(Z)$}
Let us prove the following result. 
\begin{proposition}\label{gammaZnotzero}
One can choose the constant~$c_0$ in~{\rm(\ref{defkt1})} small enough so that if~$\gamma $ belongs to~${\mathcal C}_\alpha(Z) $, then~$\gamma \cdot Z\neq 0$.
\end{proposition}
\begin{proof}
We write 
 $$ \gamma \cdot Z = \gamma \cdot \nabla_\gamma\zeta(\alpha,\gamma) + \gamma \cdot (Z-\nabla_\gamma\zeta(\alpha,\gamma))\,,$$  and observing that thanks to homogeneity arguments  $\gamma \cdot \nabla_\gamma\zeta(\alpha,\gamma)= \zeta(\alpha,\gamma)$, we deduce that for any~$\ds \gamma \in {\mathcal C}_\alpha(Z) $
 $$\left| \gamma \cdot Z \right| \geq \left| \zeta(\alpha,\gamma)\right|-  \left| \gamma \right| \left| Z-\nabla_\gamma\zeta(\alpha,\gamma) \right| \, .$$ 
 Since as argued above, $\gamma$ belongs to a fixed ring and $\zeta(\alpha,\lambda)=0$ if and only if $\lambda =0$ (as noticed in Section~\ref{freq}),  there is a positive constant $c$ such that for any $\ds \gamma \in {\mathcal C}_\alpha(Z) $
 $$ \left| \zeta(\alpha,\gamma)\right| \geq m c \, ,$$ 
 which implies in view of the definition of $\ds  {\mathcal C}_\alpha(Z) $ that there is a positive constant $\wt c$ depending only  on the ring $ {\mathcal C}$ such that 
 $$\left| \gamma \cdot Z \right| \geq  m c - m c_0  \, \wt c  \,.$$  
This ensures the desired result, by choosing the constant $c_0$ in the definition of $k_t^1$   smaller than~$    c / {  \wt c}$. 
Proposition~\ref{gammaZnotzero} is proved.
\end{proof}

 \subsubsection{A change of variables: the diffeomorphism ${\mathcal H}$}
We can assume without loss of generality (if not the integral is zero) that ${\mathcal C}_\alpha(Z) $ is not empty, and    in view of Proposition~\ref{gammaZnotzero}, we can write for any~$\gamma \in\ {\mathcal C}_\alpha(Z) $ 
the following orthogonal decomposition (since~$Z \neq 0$):
\begin{equation}\label{chgtvariable}
 \frac 1 m \, \nabla_\gamma \zeta (\alpha, \gamma)=\wt \Gamma_1\, \hat Z_1+\wt \Gamma'\, , \quad \, \hbox{with}\, \quad
\wt \Gamma_{1}:=  \Bigl(\frac 1 m \, \nabla_\gamma \zeta (\alpha, \gamma)\Big|\hat Z\Bigr)\, \, \hbox{and}\, \, \hat Z_1 := \frac Z {|Z|}\, \cdot
\end{equation}
Since~$\wt \Gamma'$ is orthogonal to the vector~$Z$, we infer
that
\begin{equation}\label{bound}
 \big |\nabla_\gamma( \Phi_\alpha ( Z ,  \frac \gamma m )) \big | = \frac 1 m \, \left|Z- \nabla_\gamma \zeta (\alpha, \gamma)\right| \geq  |\wt \Gamma'|\, .
\end{equation}
Let us consider in $ \R^p$ an orthonormal basis   $(\hat Z_1 , \dots ,\hat Z_p )$. Thanks to Proposition~\ref{gammaZnotzero}, we have~$\gamma\cdot\hat  Z_1\not=0$ on the support of the integral defining $I_\alpha(Z)$. Obviously,  the vector $\wt \Gamma'$ defined by~\eqref{chgtvariable} belongs to the vector space generated by~$(\hat Z_2 , \dots ,\hat Z_p )$. To investigate  the integral~$I_{\alpha}(Z)$ defined in~(\ref{defIalpha}), let us consider the map~$  {\mathcal H}:\gamma \mapsto \wt \gamma' $ defined by
\begin{equation}\label{defchgtvariable}
 {\mathcal C}_\alpha(Z)  \ni \gamma  \longmapsto  {\mathcal H} (\gamma):= (\gamma \cdot\hat  Z_1)\, \hat Z_1 +\sum^{p}_{k=2}(\wt \Gamma' \cdot\hat  Z_k)\, \hat Z_k=:  \sum^{p}_{k=1}\wt \gamma'_k\, \hat Z_k\, . 
\end{equation} 
\begin{proposition}\label{diffeo}
The map~${\mathcal H}$ realizes a diffeomorphism from~$ {\mathcal C}_\alpha(Z)$
into a fixed compact set of~$ \R^p$.\end{proposition}
\begin{proof}
 It is clear that the smooth function ${\mathcal H}$ maps ${\mathcal C}_\alpha(Z) $ into a fixed compact set ${\mathcal K}$ of~$ \R^p$ and that 
 $$
 \wt \gamma'_1= \gamma \cdot \hat Z_1 \, , \,\quad  \mbox{and for} \, \quad  2 \leq k \leq p \, , \, \, \, \ds \wt \gamma'_k= \frac 1 m \, \nabla_\gamma \zeta (\alpha, \gamma) \cdot  \hat Z_k \, .
 $$

 $ $

\noindent Now let us prove that thanks to Assumption~{\ref{keyp}},  the map ${\mathcal H}$ constitutes a diffeomorphism. Indeed, by straightforward computations  we find that~$D{\mathcal H}$, the differential of ${\mathcal H}$ satisfies:
$$
\begin{aligned}\langle D{\mathcal H}(\gamma)  \hat Z_1, \hat  Z_1\rangle &= 1    \\
\langle D{\mathcal H}(\gamma)  \hat Z_1,  \hat Z_k\rangle &= \Big\langle \frac 1 m D_\gamma^2 \zeta (\alpha, \gamma)  \hat Z_1,  \hat Z_k \Big\rangle \quad \hbox{for} \, \, 2 \leq  k \leq p \,, \\
\langle D{\mathcal H}(\gamma)  \hat Z_j,  \hat Z_k \rangle &= \Big\langle \frac 1 m D_\gamma^2 \zeta (\alpha, \gamma)  \hat Z_j,  \hat Z_k \Big\rangle  \quad  \hbox{for}   \, \, 2 \leq j, k \leq p \quad  \hbox{and}\\
\langle D{\mathcal H}(\gamma)  \hat Z_j,  \hat Z_1 \rangle &= 0  \quad \hbox{for}   \, \, 2 \leq j \leq p\,.
\end{aligned}$$
Proving  that ${\mathcal H}$ is a diffeomorphism amounts to showing that for any $\ds \gamma \in {\mathcal C}_\alpha(Z) $, the kernel of~$D{\mathcal H}(\gamma)$ reduces to $\{0\}$. In view of the above formulas, if~$ \ds V= \sum^{p}_{j=1} V_j  \hat Z_j$ belongs to the kernel of~$D{\mathcal H}(\gamma)$ then~$V_1= V \cdot  \hat Z_1 =0$ and~$D_\gamma^2\zeta(\alpha,\gamma)V\cdot \hat  Z_k=0$ for $2\leq k\leq p$. Thus we can write~$D^2_\gamma\zeta(\alpha,\gamma)V=\tau  \hat Z_1$ for some $\tau\in\R$. Let us point out  that since the function $\zeta(\alpha,\cdot)$ is  homogeneous of degree $1$, then $D_\gamma^2 \zeta (\alpha, \gamma) \gamma = 0$.  We deduce that 
$$0=D^2_\gamma\zeta(\alpha,\gamma)\gamma \cdot V= \gamma \cdot D^2_\gamma\zeta(\alpha,\gamma) V =\tau\, \gamma \cdot   \hat Z_1\,.$$
Since for all $\ds \gamma \in {\mathcal C}_\alpha(Z) $, $\gamma\cdot \hat  Z_1\not=0$,  we find that
 $\tau=0$ and therefore  $D^2_\gamma\zeta(\alpha,\gamma)V=0$.
But Assumption~\ref{keyp} states that the Hessian $  \ds D_\gamma^2 \zeta (\alpha, \gamma)$ is of rank $p-1$, so we conclude that  $V$ is colinear  to  $\gamma$. But we have seen that~$V \cdot  \hat Z_1 =0$, which contradicts the fact that~$\gamma \cdot  \hat Z_1\neq 0$. This  entails that~$V$ is null and ends the proof of the proposition.  \end{proof}

  \medbreak 
    \noindent  We can therefore perform   the change of variables defined by \eqref{defchgtvariable} in the right-hand side of~\eqref{usformula}, to obtain $$ |I_{\alpha}(Z)|  
  \leq   C_N \,m^{ N } \int_{{\mathcal K}}
\frac 1 {\left(1+t|\wt \gamma'|^2\right)^{N}}\, d\wt \gamma'\,d \wt \gamma_1 \, .
$$

\subsubsection{End of the proof: convergence of the series}
Choosing $ N =  p-1$ implies by  the
change of variables~$ \gamma^\sharp = t^{\frac 1 2} \wt \gamma'$ that  there is a constant~$C$ such that
$$
|I_{\alpha}(Z)|  \leq C |t|^{-\frac{p-1}2}  \, m^{p-1} \, ,
$$
which gives rise to 
$$
\Big| \int_{ {\mathcal C}_\alpha ( Z ) }    {\rm e}^{it \Phi_\alpha ( Z,  \frac \gamma m )}  \psi_\alpha(\gamma) d \gamma\Big|\leq 
 C |t|^{-\frac{p-1}2}   m^{p-1 }.
$$
We get exactly in the same way that
$$
\Big | \int_{ |\nabla_\lambda \Phi_0 ( Z ,  \lambda )|\leq c_0 }  {\rm e}^{it\Phi_0(Z, \lambda)}  G_0 \big (P,Q,\eta(\lambda)\big)  |{\mbox {Pf}} (\lambda) |\, F (\lambda)\, d \lambda \Big |  \leq C |t|^{-\frac{p-1}2}  .
$$
Finally returning to the kernel~$k^1_t$ defined in~(\ref{defkt1}), we get
$$
\begin{aligned}
|k^1_t(P,Q,tZ)| & \leq C   |t|^{-\frac{p-1}2}  + C\,  |t|^{-\frac{p-1}2}\, \sum_{m \in {\mathbb N^*}} m^{d-1}m^{-d-p}  \, m^{p-1}  \\
& \leq C  |t|^{-\frac{p-1}2}  \, ,
\end{aligned}
$$
since the series over~$m$ is convergent. The dispersive estimate  is thus proved for $k^1_t$.\\

 \subsection{Stationary phase argument for $k_t^2$}

\noindent We now prove~(\ref{goal}) for   $k^2_t $, which is easier since the gradient of the phase is bounded from below. We claim that  there is a constant $C$ such that 
\begin{equation}\label{estphas1}
|k^2_t (P,Q,tZ) | \leq \frac {C} {\,t^{\frac{p-1}2} }\, \cdot \end{equation}
This can be achieved as above by means of  adequate integrations by parts. Indeed,  in the case when~$\alpha \neq 0$, consider the following first order operator:
$$
{\mathcal L}_\alpha^2:= - i \frac{ \nabla_\gamma( \Phi_\alpha ( Z ,  \frac \gamma m )) \cdot \nabla_\gamma}{| \nabla_\gamma( \Phi_\alpha ( Z ,  \frac \gamma m ))|^2} \, \cdot$$
Note that when $\alpha=0$, the arguments are the same without performing the change of variable $\lambda={\gamma/ m}$.\\
The operator ${\mathcal L}^2_\alpha$ obviously satisfies 
$${\mathcal L}_\alpha^2 \,   {\rm e}^{it  \Phi_\alpha ( Z,  \frac \gamma m )}= t \,  {\rm e}^{it \,  \Phi_\alpha ( Z,  \frac \gamma m )}\,  ,$$ 
hence by repeated integrations by parts, we get
$$
\begin{aligned}
J_\alpha(P,Q,tZ) &:= \int_{ | \nabla_\gamma( \Phi_\alpha ( Z ,  \frac \gamma m ))|\geq c_0 }    {\rm e}^{it \Phi_\alpha ( Z,  \frac \gamma m )} \psi_{\alpha}(\gamma)\, d \gamma \\
&= \frac{ 1}{t^N}\,   \int_{ | \nabla_\gamma( \Phi_\alpha ( Z ,  \frac \gamma m ))|\geq c_0 }  {\rm e}^{it \Phi_\alpha ( Z,  \frac \gamma m ))} ({}^t{\mathcal L}_\alpha^2)^N \psi_{\alpha}(\gamma)\, d\lambda \, .
\end{aligned}
$$
Let us admit the following lemma for a while. 
\begin{lemma}\label{transposeofI}
For any integer~$N$, 
there is   a  smooth function~$\theta_N$  compactly supported  on a compact set of~$\R^p$ such that
$$
\big|(^t{} {\mathcal L}_\alpha^2  )^N\psi_{\alpha}(\gamma) \,  \big| \leq  \frac{\theta_N(\gamma) \, m^{N } }{| \nabla_\gamma( \Phi_\alpha ( Z ,  \frac \gamma m ))|^{N } }\, \cdot
$$
\end{lemma}
\noindent  One then observes that if $\gamma$ is in  the support of the integral defining $k^2_t$, the lemma implies
$$
\big|(^t{} {\mathcal L}_\alpha^2  )^N\psi_{\alpha}(\gamma) \,  \big| \leq  \frac{\theta_N(\gamma)}{c_0^N}  \, m^{N}.
$$
This estimate   ensures the result as in Section~\ref{statphas} by taking $N=p-1$.

\subsection{Proofs of     Lemma {\bf\ref{transposeofL}} and Lemma {\ref{transposeofI}}}\label{prooflemmas}
  
Lemma~\ref{transposeofL}  is an obvious consequence of the following Lemma~\ref{transposeofLbis}, taking~$(a,b)\equiv (0,0)$. We omit the  proof of Lemma~{\rm\ref{transposeofI}} which consists  in a straightforward  modification of the arguments developed below. 

\medskip

\begin{lemma}\label{transposeofLbis}
For any integer~$N$,  one can write 
$$
(^t{} {\mathcal L}_\alpha^1  )^N\psi_{\alpha}(\gamma) = f_{N,m} \big(\gamma,t^\frac12 \nabla_\gamma( \Phi_\alpha ( Z ,  \frac \gamma m ))\big ) \, ,
$$
with $|\alpha|=m$, and where~$f_{N,m}$ is a smooth  function  supported on~${\mathcal C} \times \R^{p}$ with~$\mathcal C$ a fixed ring of~$ \R^{p}$,  such that for any couple~$(a,b) \in {\mathbb N}^p \times {\mathbb N}^{p}$, there is a constant~$C$ (independent of $m$) such that
$$
|\nabla_\gamma^a  \nabla_\Theta^b  f_{N,m} (\gamma,\Theta) | \leq C \, m^{N+|a| } (1+|\Theta|^2)^{-N - \frac{ |b|}{2}} \, .
$$
\end{lemma}
\begin{proof}[Proof of Lemma~{\rm\ref{transposeofLbis}}]
Let us prove the result by induction over~$N$. 
 We start with the case when~$N$ is equal to zero. Notice that in that case the function~$ f_{0,m} (\gamma,\Theta) = \psi_{\alpha}(\gamma) $ does not depend on the quantity~$\Theta = t^\frac12 \nabla_\gamma (\Phi_\alpha ( Z,  \frac \gamma m) )$, so  we need to check that  for any~$a \in {\mathbb N}^p$, there is a constant~$C$  such that
\begin{equation}\label{formula}
|\nabla_\gamma^a  \psi_{\alpha}(\gamma) | \leq C \,  m^{|a| }\, ,
\end{equation}
when $|\alpha|=m$. The case when~$a=0$ is obvious thanks  to the uniform bound on~$G_\alpha$. To deal with the case~$|a| \geq 1 $, we state the following technical result which will be proved at the end of this paragraph.
\begin{lemma}\label{hermitederivatives}
For any integer~$k$,  there is a constant~$C $ such that
the following bound holds for the function~$g_n$ defined in~{\rm(\ref{defgxi1xi2})}, $n\in{\mathbb N}$:
$$
\forall (\xi_1,\xi_2) \in \R^2 \, , \quad \big | ( \xi_1 \partial_{\xi_1}+\xi_2 \partial_{\xi_2} )^k g_n (\xi_1, \xi_2) \big | \leq C n^k \, .
$$
\end{lemma}
\noindent Let us now compute~$\nabla_\gamma^a  \psi_{\alpha}(\gamma)$. Recall  that according to~(\ref{def:psialpha}),
$$
\begin{aligned}
 \psi_{\alpha}(\gamma) &= G_\alpha\big (P,Q,\widetilde \eta_m(\gamma)\big ) F \Big(\frac \gamma m\Big)  \,  |{\mbox {Pf}} (\gamma) |\\
 &= F\left({\gamma\over m}\right)\,  \prod_{j=1}^d \psi_{\alpha,j}(\gamma)\, ,
\end{aligned}
$$
where $$
\psi_{\alpha,j}(\gamma)
:= \eta_j(\gamma)\widetilde \theta \big ((2\alpha_j+1)  \widetilde  \eta_{j,m}(\gamma)\big) \, g_{\alpha_j} \big(\sqrt {   \widetilde \eta_{j,m}(\gamma)  }P_j,\sqrt {   \widetilde \eta_{j,m}(\gamma)  }Q_j\big) \, , \quad \widetilde\eta_{j,m}(\gamma) := \frac1m \eta_{ j}(\gamma) \, .
$$
We compute
$$
\begin{aligned}
\nabla_\gamma^a  \psi_{\alpha,j}(\gamma) &= \sumetage{b \in {\mathbb N}^p}{0 \leq |b| \leq |a|} \binom{b}{a} \nabla_\gamma^b   \big( \theta \big ((2\alpha_j+1)  \widetilde \eta_{j,m}(\gamma)\big) \big)\nabla_\gamma^{a-b} \big( \eta_j(\gamma)  g_{\alpha_j} \big(\sqrt {   \widetilde \eta_{j,m}(\gamma)  }P_j,\sqrt {   \widetilde \eta_{j,m}(\gamma)  }Q_j\big)\big) \, .
\end{aligned}
$$
 Let us assume first that~$|a-b| = 1$. Then we write, for some~$1 \leq \ell \leq p$,
$$
\begin{aligned}
\partial_{\gamma_\ell} \Big(  \eta_j(\gamma) g_{\alpha_j} \big(\sqrt {    \widetilde \eta_{j,m}(\gamma)  }P_j&,\sqrt {   \widetilde \eta_{j,m}(\gamma)  }Q_j\big)\Big)= 
\partial_{\gamma_ \ell}  \eta_j(\gamma) g_{\alpha_j}  \big(\sqrt {   \widetilde \eta_{j,m}(\gamma)  }P_j,\sqrt {   \widetilde \eta_{j,m}(\gamma)  }Q_j\big) \\
& \quad + \eta_j(\gamma)
\frac{
\partial_{\gamma_ \ell} \widetilde \eta_{j,m}(\gamma) }{ 2   {  \widetilde \eta_{j,m}(\gamma)}  }  \times \big( (\xi_1 \partial_{\xi_1} + \xi_2\partial_{\xi_2}) g_{\alpha_j}   \big)\Big(\sqrt {   \widetilde \eta_{j,m}(\gamma)  }P_j,\sqrt {   \widetilde \eta_{j,m}(\gamma)  }Q_j\Big) \, .
\end{aligned}
$$
Next we use the fact that there is a constant~$C$ such that on the support of~$  \theta \big ((2\alpha_j+1)  \widetilde \eta_{j,m}(\gamma)\big)$, 
$$
 \widetilde \eta_{j,m}(\gamma) \geq \frac1{Cm}\quad \mbox{and} \quad \left| 
\partial_{\gamma_ \ell} \widetilde \eta_{j,m}(\gamma) \right|
\leq  \frac C m\,,$$ 
so applying Lemma~\ref{hermitederivatives}   gives
$$
|\nabla_\gamma  \psi_{\alpha,j}(\gamma) | \lesssim    { \alpha_j  }\,. 
$$
Recalling that~$\alpha_j  \leq m$ and that, for all $i \in \{1,\cdots,d\}$, $\psi_{\alpha,j}$ is uniformly bounded, this easily achieves the proof of Estimate  \eqref{formula} in the case~$|a|=1$ by taking the product over~$j$. Once we have noticed that 
$$\alpha^{a_1}_1 \cdots \alpha^{a_d}_d \lesssim (\alpha_1+\cdots+\alpha_d)^{a_1+\cdots+a_d}\, ,$$ the general case (when~$|a|>1$) is dealt with identically, we omit the  details.

\medskip
\noindent
Finally let us proceed with the induction:   assume that for some integer~$N$ one can write
$$
(^t{} {\mathcal L}_\alpha^1 )^{N-1}\psi_{\alpha}(\gamma)  = f_{N-1,m} \big(\gamma,t^\frac12 \nabla_\gamma( \Phi_\alpha ( Z ,  \frac \gamma m ))\big )
$$
where~$f_{N-1,m}$ is a smooth  function  supported on~${\mathcal C} \times \R^{p}$, such that for any couple~$(a,b) \in {\mathbb N}^p \times {\mathbb N}^{p}$, there is a constant~$C$ (independent of $m$) such that
\begin{equation}\label{induction}
|\nabla_\gamma^a   \nabla_\Theta^b f_{N-1,m} (\gamma,\Theta) | \leq C \, m^{ {N -1 +  |a|}} (1+|\Theta|^2)^{- (N-1) - \frac{ |b|}{2}} \, .
\end{equation}
We compute   for any function~$\Psi(\gamma)$,
\begin{eqnarray*}
^t{} {\mathcal L}_\alpha^1 \Psi(\gamma)& = & i \frac{  \nabla_\gamma( \Phi_\alpha ( Z ,  \frac \gamma m ))\, \cdot \nabla_\gamma \Psi(\gamma)}{  1+t | \nabla_\gamma( \Phi_\alpha ( Z ,  \frac \gamma m ))|^2 }+  \frac{ 1 + i \Delta ( \Phi_\alpha ( Z ,  \frac \gamma m ))}{  1+t | \nabla_\gamma( \Phi_\alpha ( Z ,  \frac \gamma m ))|^2 }\Psi(\gamma) \\ 
&-& 2it \sum_{1\leq j, k\leq p} \frac{ \partial_{\gamma_j} \partial_{\gamma_k}  ( \Phi_\alpha ( Z ,  \frac \gamma m ))\partial_{\gamma_j}   ( \Phi_\alpha ( Z ,  \frac \gamma m ))\partial_{\gamma_k} ( \Phi_\alpha ( Z ,  \frac \gamma m )) }{  (1+t |\nabla_\gamma  ( \Phi_\alpha ( Z ,  \frac \gamma m ))|^2)^2 }\Psi(\gamma) \, .
\end{eqnarray*}
We apply that formula to~$\Psi := f_{N-1} \big(\gamma,t^\frac12 \nabla_\gamma  ( \Phi_\alpha ( Z ,  \frac \gamma m ))\big )
 $ and  estimating each of the three terms separately we find (using the fact that~$m \geq 1$),
 $$
 \begin{aligned}
\Big | ^t{} {\mathcal L}_\alpha^1 \Big ( f_{N-1} \big(\gamma,t^\frac12 \nabla_\gamma  ( \Phi_\alpha ( Z ,  \frac \gamma m ))\big )\Big)\Big | &\leq C \big( 1+t | \nabla_\gamma( \Phi_\alpha ( Z ,  \frac \gamma m ))|^2 \big)^{-1} \\
& \quad \times m^{ {N -1 +1 }} (1+t | \nabla_\gamma( \Phi_\alpha ( Z ,  \frac \gamma m ))|^2)^{- (N-1) }  \\
&    + C \big( 1+t | \nabla_\gamma( \Phi_\alpha ( Z ,  \frac \gamma m ))|^2 \big)^{-1} \\
 & \quad \times m^{ {N -1  }} (1+t | \nabla_\gamma( \Phi_\alpha ( Z ,  \frac \gamma m ))|^2)^{- (N-1) } \\
 &    + C t \, | \nabla_\gamma( \Phi_\alpha ( Z ,  \frac \gamma m ))|^2 \big( 1+t | \nabla_\gamma( \Phi_\alpha ( Z ,  \frac \gamma m ))|^2 \big)^{-2} \\
 & \quad \times m^{ {N -1 }} (1+t | \nabla_\gamma( \Phi_\alpha ( Z ,  \frac \gamma m ))|^2)^{- (N-1) }
 \end{aligned}
 $$
 thanks to   the induction assumption~(\ref{induction}) along with~\eqref{formula} and the fact that on $ {\mathcal C}_\alpha(Z)$,  all the derivatives of the function $\nabla_\gamma\big(  \Phi_\alpha \big( Z ,  \frac \gamma m \big)\big)$ are  uniformly bounded with respect to $\alpha$ and $Z$.  A similar argument allows to control derivatives in~$\gamma$ and~$\Theta$, so Lemma~{\rm\ref{transposeofLbis}} is proved. \end{proof}

 \medskip
 
\begin{proof}[Proof of Lemma~{\rm\ref{hermitederivatives}}]
By definition of~$g_n$ and using the change of variable
$$
\xi \mapsto \xi - \frac{\xi_1}2
$$
we recover the Wigner-type formula
$$
g_n(\xi_1,\xi_2) = \int_{\R} e^{-i \xi_2 \xi} h_n \big (\xi +  \frac{\xi_1}2
\big) h_n \big (\xi - \frac{\xi_1}2
\big) \, d\xi \, .
$$
Then an easy computation shows that for all~$k$
$$
|( \xi_1 \partial_{\xi_1} + \xi_2 \partial_{\xi_2} )^kg_n(\xi_1,\xi_2)  |  
\leq  \int_{\R}\Big |( \xi_1 \partial_{\xi_1} + \xi \partial_{\xi} + 1 )^k \Big( h_n \big (\xi +  \frac{\xi_1}2
\big) h_n \big (\xi - \frac{\xi_1}2
\big)\Big) \Big | \, d\xi  \, .
$$
By the Cauchy-Schwarz inequality (and a change of variables to transform~$\xi +  \frac{\xi_1}2$ and~$\xi - \frac{\xi_1}2$ into~$(\xi,\xi')$), it remains therefore to check that for all~$k$
$$
\|(\xi\partial_{\xi})^k h_n\|_{L^2(\R)} \leq C_k n^k \, .
$$
This again reduces to checking that
\begin{equation}\label{finalestimatereferee}
\|\xi^{2k} h_n\|_{L^2(\R)} + \| h_n^{(2k)}\|_{L^2(\R)}  \leq  C_k  n^k \, .
\end{equation}
This estimate is a consequence of the   identification of the domain of~$\sqrt H$
$$
D(\sqrt H ) = \Big \{ u \in L^2(\R) \, / \, \xi u \, ,   u' \in L^2(\R)\Big \}
$$
which classically extends to powers of~$\sqrt H$
$$
D(H^\frac p2 ) = \Big \{ u \in L^2(\R) \, / \, \xi ^{p-\ell}u^{(\ell)} \in L^2(\R) \, , 0 \leq \ell \leq p\Big \} \, .
$$
 Then~(\ref{finalestimatereferee}) is finally obtained by applying this to~$p=2k$, recalling that~$H^k h_n = (2n+1)^k h_n$.  The proposition is proved.
\end{proof}

 \medskip
%%%%%%%%%%%%%%%%%%%%%%%%%%%%%%%%%%%%%%%%%%%%%

\section{Optimality of the dispersive estimates}\label{optimality} 

 \noindent   In this section, we first end the proof of Theorem~\ref{dispgrad} by proving the optimality of the dispersive estimates for groups satisfying Assumption~\ref{keyp}. Then, we prove Proposition~\ref{remnodisp}.

\subsection{Optimality for groups satisfying Assumption~\ref{keyp}}

  Let us now end the proof of Theorem~\ref{dispgrad} by establishing the optimality of the dispersive estimate~\eqref{eq:gradeddispS}. We use the fact that there always exists $\lambda^*\in\Lambda$ such that
  \begin{equation}\label{prop:lambda0}
  \nabla_\lambda \zeta(0,\lambda^*)\not=0,
  \end{equation}
  where the function $\zeta$ is defined in~(\ref{eq:freqxy}). Indeed, if not, the map $\lambda\mapsto \zeta(0,\lambda)$ would be constant which is in contradiction with the fact that $\zeta$ is homogeneous of degree~$1$. 
We   prove the following proposition, which  yields the optimality of the dispersive estimate of Theorem~\ref{dispgrad}.   
  \begin{proposition}\label{prop:optimal}
  Let $\lambda^*\in\Lambda$ satisfying~{\rm(\ref{prop:lambda0})}.
  There is a function $g\in{\mathcal D}(\R^p)$   compactly supported in a connected open neighborhood of $\lambda^*$ in $\Lambda$,  such that   for the initial data~$f_0$   defined by
\begin{equation}\label{deff0optimality}
\forall (\lambda,\nu) \in\mathfrak r(\Lambda),\;\;  {\mathcal F}(f_0) (\lambda, \nu)  h_{\alpha,\eta(\lambda)} = 0 \, \, \hbox{for} \, \, \alpha \neq 0 \, \, \hbox{and} \, \, {\mathcal F}(f_0)  (\lambda, \nu)  h_{0,\eta(\lambda)} = g(\lambda) h_{0,\eta(\lambda)}\, ,
\end{equation}
 there exists $c_0> 0$ and $x\in G$  such that  
   $$| {\rm e}^{-it\Delta_G} f_0(x)| \ge  c_0\, |t|^{-{k+p-1\over 2}} \, .$$ 
  \end{proposition} 
    \begin{proof} Let~$g$ be any smooth compactly supported function over~$\R^p$, and define~$f_0$ by~(\ref{deff0optimality}). For any point $x= e^X\in G$ under the form $X=(P=0,Q=0,Z,R)$, the   inversion formula gives
$${\rm e}^{-it\Delta_G}f_0(x)= \kappa \, \int_{\lambda\in\Lambda}\int_{\nu\in\mathfrak r^*_\lambda} {\rm e}^{it|\nu|^2+it\zeta(0,\lambda)-i\lambda(Z)-i\nu(R)}
g(\lambda)|{\mbox {Pf}} (\lambda)|d\nu d\lambda \, .$$
To simplify notations, we set $\zeta_0(\lambda):=\zeta(0,\lambda)$.
 Setting $Z=tZ^*$ with~$Z^*:=\nabla_\lambda\zeta(0,\lambda^*)\not=0$, we get as  in~(\ref{schrodispersion})
 $$
 \Big |
 {\rm e}^{-it\Delta_G}f_0(x) 
 \Big | = c_1 |t|^{-{k\over 2}}  
 \Big |
\int_{\lambda\in\R^p} {\rm e}^{it (\lambda\cdot Z^*-\zeta_0(\lambda))}
g(\lambda)|{\mbox {Pf}}(\lambda)| d\lambda
 \Big |
 \, $$
 for some constant $c_1>0$.
 Without loss of generality, we can assume 
 $$
 \lambda^*=(1,0,\dots,0)
 $$ (if not, we perform a change of variables~$\lambda\mapsto \Omega\lambda$ where $\Omega$ is a fixed orthogonal matrix), and we now shall  perform a stationary phase in the variable $\lambda'$, where we have written~$\lambda=(\lambda_1,\lambda')$. For any fixed~$\lambda_1$, the phase 
$$\Phi_{\lambda_1}(\lambda',Z):= Z\cdot\lambda -\zeta_0(\lambda)$$
has a stationary point $\lambda'$ if and only if   $Z'=\nabla_{\lambda'}\zeta_0(\lambda)$ (with the same notation~$Z=(Z_1,Z')$). We observe that the homogeneity of the function $\zeta_0$ and the definition of $Z^*$ imply  that 
$$Z^*=\nabla_\lambda \zeta_0(1,0,\dots,0)=\nabla_\lambda\zeta_0(\lambda_1,0,\dots,0),\;\;\forall \lambda_1\in\R \, ,
$$
hence if $\lambda'=0$, then  the phase $\Phi_{\lambda_1} (0,Z^*)$ has a stationary point.

\smallskip
\noindent From now on we choose~$g$ supported near those stationary points~$(\lambda_1,0)$, and vanishing in the neighborhood of any other stationary point.

\smallskip
\noindent Let us now study the Hessian of~$\Phi_{\lambda_1}$ in $\lambda'=0$. Again because of the homogeneity of the function $\zeta_0$, we have 
$$\left[{\rm Hess} \,\zeta_0(\lambda)\right]\lambda=0,\;\;\forall \lambda\in\R^p.$$
In particular, for all $\lambda_1\not=0$,  ${\rm Hess}\,\zeta_0(\lambda_1,0,\cdots,0)\left(\lambda_1,0,\cdots,0\right)=0$ and the matrix ${\rm Hess}\,\zeta_0(\lambda_1,0,\cdots,0)$ in the canonical basis is of the form 
$${\rm Hess}\,\zeta_0(\lambda_1,0,\cdots,0)=\begin{pmatrix}  0 & 0 \\ 0 & {\rm Hess}_{\lambda',\lambda'}\,\zeta_0(\lambda_1,0,\cdots,0)\end{pmatrix}.$$
Using that ${\rm Hess}\,\zeta_0(\lambda_1,0,\cdots,0)$ is of rank $p-1$, we deduce that $ {\rm Hess}_{\lambda',\lambda'}\,\zeta_0(\lambda_1,0,\cdots,0)$  is also of rank~$p-1$ and we conclude by   the stationary phase theorem (\cite{stein}, Chap. VIII.2), choosing~$g$ so that the remaining integral in~$\lambda_1$ does not vanish.

\end{proof}

%%%%%%%%%%%%%%%%%%%%%%%%%%%%%%%%%%%%%%%%%%%%%%%%%%%%%%

\subsection{Proof of Proposition~\ref{remnodisp}}

Assume that $G$ is a  step~2 stratified  Lie group whose radical  index is null and for which $\zeta( 0,\lambda)$ is a  linear form  on    each connected component of the Zariski-open subset~$\Lambda$. 
Let~$g$ be a smooth nonnegative function supported in one of the  connected components of $\Lambda$ and define~$f_0$  by
$$  {\mathcal F}(f_0) (\lambda) h_{\alpha,\eta(\lambda)} = 0 \, \, \hbox{for} \, \, \alpha \neq 0 \, \, \hbox{and} \, \, {\mathcal F}(f_0) (\lambda) h_{0,\eta(\lambda)} = g(\lambda) h_{0,\eta(\lambda)}\, .$$
By the inverse Fourier formula, if $x=e^X\in G$ is such that $X=(P=0,Q=0,tZ)$, then we have 
$$  {\rm e}^{-it\Delta_G} (x) =  \kappa \, \int\,{\rm e}^{-i t\,\lambda(Z)}   {\rm e}^{it  \zeta(0, \lambda)}g(\lambda)  \,|{\mbox {{\mbox {Pf}}}} (\lambda) | \, d\lambda \, .$$
  Since $ \zeta(0, \lambda)$ is a  linear form  on each connected component of $\Lambda$, there exists  $Z_0$ in~$\mathfrak z$ such that for
  $$\forall\lambda\in \mathfrak z^*\cap {\rm supp} \,g, \;\;  -\lambda(Z_0)+  \zeta(0, \lambda) = 0\,.$$ 
As a consequence, choosing $Z=Z_0$, we obtain  
$$  {\rm e}^{-it\Delta_G} (x) =  \kappa \, \int g(\lambda)  \,|{\mbox {{\mbox {Pf}}}} (\lambda) | \, d\lambda \, \not=0,$$ 
which ends the proof of the result.

%%%%%%%%%%%%%%%%%%%%%%%%%%%%%%%%%%%%%%%%%%%%%%%%%%%%%%%%%%%%%%%%%%

 \appendix 
 \section{ On the inversion formula in Schwartz space}  
 \noindent    This section is dedicated to the proof of the inversion formula in the Schwartz space  ${\mathcal S}(G)$ (Proposition~\ref{inversioninS} page~\pageref{inversioninS}). 
\label{appendixinversion}
 \begin{proof} 
 We first observe that to establish \eqref{inversionformula},  it suffices to prove that 
\begin{equation}
\label{inverssimp}f(0)
=   \kappa \, \int_{\lambda\in\Lambda}\int_{\nu\in\mathfrak r^*_\lambda} {\rm{tr}} \, \Big({\mathcal F}(f)(\lambda,\nu)  \Big)\, |{\mbox {Pf}} (\lambda) |\, d\nu\,d\lambda \,. \end{equation}
Indeed,  introducing the auxiliary   function $g$ defined by $g(x') := f(x \cdot x')$  which obviously  belongs to~${\mathcal S}(G)$ and satisfies ${\mathcal F}(g)(\lambda,\nu)= u^{\lambda,\nu}_{X(\lambda,x^{-1})} \circ {\mathcal F}(f)(\lambda,\nu)$, and  assuming~(\ref{inverssimp}) holds, we get 
  \begin{eqnarray*} f(x)= g(0) &=& \kappa  \, \int_{\lambda\in\Lambda}\int_{\nu\in\mathfrak r^*_\lambda} {\rm{tr}} \, \Big( {\mathcal F}(g)(\lambda,\nu)  \Big)\, |{\mbox {Pf}} (\lambda) |\, d\nu\,d\lambda \\
& =& \kappa  \,  \int_{\lambda\in\Lambda}\int_{\nu\in\mathfrak r^*_\lambda} {\rm{tr}} \, \Big(u^{\lambda,\nu}_{X(\lambda,x^{-1})} {\mathcal F}(f)(\lambda,\nu)  \Big)\, |{\mbox {Pf}} (\lambda) |\, d\nu\,d\lambda
  \, , \end{eqnarray*}
which is the desired result.
 \medskip
 
 \noindent Let us now focus on~(\ref{inverssimp}). 
    In order to compute the right-hand side of Identity \refeq{inverssimp}, we introduce 
       \begin{eqnarray*}
       A&:=&
        \int_{\lambda\in\Lambda}\int_{\nu\in\mathfrak r^*_\lambda} {\rm{tr}} \, \Big( {\mathcal F}(f)(\lambda,\nu)  \Big)\, |{\mbox {Pf}} (\lambda) |\, d\nu\,d\lambda  \\
    &= &    \int_{\lambda\in\Lambda}\int_{\nu\in\mathfrak r^*_\lambda} \int_{x \in G}  \sum_{\alpha \in {\mathbb N}^d} \big( u^{\lambda,\nu}_{X(\lambda,x)} h_{\alpha,\eta(\lambda)} |  h_{\alpha,\eta(\lambda)} \big)\, |{\mbox {Pf}} (\lambda) | \, f(x) \, d\mu(x) \, d\nu\,d\lambda \, ,
       \end{eqnarray*}
       with the notation of Section~\ref{Fourier}.
 In order to carry on the calculations, we need to resort to a  Fubini argument which comes from the following  identity:
 \begin{equation}\label{hyp:inv}
 \sum_{\al \in {\mathbb N}^d} \int_{\lambda\in\Lambda}\int_{\nu\in\mathfrak r^*_\lambda}  \|{\mathcal F}(f)(\lambda,\nu) h_{\alpha,\eta(\lambda)}\|_{L^2( \mathfrak p_\lambda)}  |{\mbox {Pf}} (\lambda) |\,  d\nu\,d\lambda\, < \infty  \, .
\end{equation}
We postpone the proof of~(\ref{hyp:inv}) to the end of this section. Thanks to~(\ref{hyp:inv}), the order of integration does not matter and  we can transform the expression of $A$: we use the fact  that for any $\alpha \in {\mathbb N}^d$
  $$ \big( u^{\lambda,\nu}_{X(\lambda,x)} h_{\alpha,\eta(\lambda)} |  h_{\alpha,\eta(\lambda)} \big)=  {\rm e}^{-i\nu(R) -i\lambda( Z)}\int_{  \R^d} {\rm e}^{-i \sum_j\eta_j(\lambda)\, (\xi_j+{1\over 2} P_j)\, Q_j} 
h_{\alpha,\eta(\lambda)}(P+ \xi)\, h_{\alpha,\eta(\lambda)}(\xi)d \xi \, ,$$ 
where we have identified ${\mathfrak p}_\lambda$ with $\R^d$, and this gives rise to 
  $$
\displaylines{
  A=  \int_{\lambda\in\Lambda}\int_{\nu\in\mathfrak r^*_\lambda}  \int_{x \in G} \int_{\xi \in \R^d} \, \sum_{\alpha \in {\mathbb N}^d}  \, {\rm e}^{-i\nu( R) -i\lambda (Z)} {\rm e}^{-i \sum_j\eta_j(\lambda) \,(\xi_j+{1\over 2} P_j) \,  Q_j}  \, h_{\alpha,\eta(\lambda)}(P+ \xi)\, \cr
  {}  {}   \qquad\qquad \qquad\qquad \qquad\qquad \qquad\qquad \times \,  h_{\alpha,\eta(\lambda)}(\xi)  \,  |{\mbox {Pf}} (\lambda) | \, f(x) \, d\mu(x) \, d \xi \, d\nu\,d\lambda \, ,} 
$$
where we recall that 
$$ h_{\alpha,\eta(\lambda)}(\xi) = \prod_{j=1}^d h_{\alpha_j,\eta_j(\lambda)}(\xi_j) \quad \mbox{with} \quad  h_{\alpha_j,\eta_j(\lambda)}(\xi_j)= \eta_j(\lambda)^{\frac 1 4}  \, h_{\alpha_j} \Big(\sqrt{\eta_j(\lambda)}  \,\xi_j\Big)\, .$$
\medskip 
\noindent
  Performing the  change of variables
     $$
 \left\{
\begin{array}{l}
\wt \xi_j = \sqrt{\eta_j(\lambda)}  \,\xi_j\\
\wt P_j = \sqrt{\eta_j(\lambda)}  \,P_j\\
\wt Q_j  =  \sqrt{\eta_j(\lambda)}  \,Q_j\, 
\end{array}
\right.
$$
for $j \in \{1,\dots, d\}$, we obtain, dropping the $\,\widetilde\,$ on the variables,
 $$
\displaylines{
  A=  \int_{\lambda\in\Lambda}\int_{\nu\in\mathfrak r^*_\lambda}  \int_{(P,Q,R,Z)\in\R^{n}} \int_{\xi \in \R^d} \, \sum_{\alpha \in {\mathbb N}^d}  \, {\rm e}^{-i\nu( R) -i\lambda ( Z)} {\rm e}^{-i \sum_\ell (\xi_\ell+{1\over 2} P_\ell) \cdot Q_\ell}  \, \prod_{j=1}^d \, h_{\alpha_j}(P_j+ \xi_j)\, h_{\alpha_j}(\xi_j) \cr
  {}  {}   \quad \,  \times  \, f\big(\eta^{-{1\over 2}}(\lambda)\,P, \eta^{-{1\over 2}}(\lambda) \,Q, R, Z\big) \, dP  \,dQ \,dR\,dZ\, d \xi \, d\nu\,d\lambda \, ,} 
$$
with $\eta^{-{1\over 2}}(\lambda)  \,P:=(\eta_1^{-{1\over 2}}(\lambda) \,P_1, \dots, \eta_d^{-{1\over 2}}(\lambda) \,P_d)$ and similarly for $Q$.

\bigskip 
\noindent
Then  using the  change of variables $\xi_j'= \xi_j+P_j$, for $j \in \{1,\dots, d\}$, gives
$$
\displaylines{
  A=  \int_{\lambda\in\Lambda}\int_{\nu\in\mathfrak r^*_\lambda}  \int_{(\xi',Q,R,Z)\in\R^n} \int_{\xi \in \R^d} \, \sum_{\alpha \in {\mathbb N}^d}  \, {\rm e}^{-i\nu(R) -i\lambda  (Z)} {\rm e}^{-{i\over 2} \sum_\ell (\xi_\ell+\xi_\ell' ) \cdot Q_\ell}  \, \prod_{j=1}^d \, h_{\alpha_j}( \xi'_j)\, h_{\alpha_j}(\xi_j) \cr
  {}  {}   \quad \,  \times  \, f\big(\eta^{-{1\over 2}}(\lambda)\,(\xi'-\xi), \eta^{-{1\over 2}}(\lambda) \,Q, R, Z\big)   \,d\xi'  \, dQ \,dR\,dZ\, d \xi \, d\nu\,d\lambda \, .} 
$$
\medskip 
\noindent
Because $(h_{\alpha})_{\alpha\in\N^d}$ is a Hilbert basis of $L^2(\R^d),$    we have for all $\varphi \in L^2(\R^d)$ 
$$ \varphi(\xi) = \sum_{\alpha \in {\mathbb N}^d}  \, \int_{\xi' \in \R^d} \varphi(\xi') \, h_{\alpha}(\xi') \,d\xi' \,h_{\alpha}(\xi) \,,$$ 
which leads to 
$$  A=  \int_{\lambda\in\Lambda}\int_{\nu\in\mathfrak r^*_\lambda}  \int_{(Q,R,Z)\in\R^{d+k+p}} \int_{\xi \in \R^d}  \, {\rm e}^{-i\nu( R) -i\lambda (Z)} {\rm e}^{-i \xi \cdot Q} \, f\big(0, \eta^{-{1\over 2}}(\lambda)  \,Q, R, Z\big)    \, dQ \,dR\,dZ\, d \xi \, d\nu\,d\lambda \, .$$ 
Applying the  Fourier inversion formula successively on $\R^d$,  $\R^k$ and on $\R^p$ (and identifying ${\mathfrak r}(\Lambda)$ with~$\R^p\times\R^k$),  we conclude that there exists a constant $\kappa>0$ such that  
$$ A=   \kappa \, f(0)  \,,$$
which ends the proof of \eqref{inverssimp}.
\medskip

\noindent Let us conclude the proof by showing~(\ref{hyp:inv}). We choose
$M$  a nonnegative integer.
According to the obvious fact that the
function~$(\mbox{Id}-\Delta_{G})^M f$ also belongs to~${\mathcal S}(G)$ (hence to $L^1(G)$), we get in view of Identity \eqref{formulafourierdelta} 
$$ {\mathcal F}(f)(\lambda,\nu)h_{\alpha,\eta(\lambda)}= \Big(1+ |\nu|^2+\zeta(\alpha,\lambda)\Big)^{-M} {\mathcal  F} \left( (\mbox{Id}-\Delta_{G})^M f\right)(\lam,\nu)  h_{\alpha,\eta(\lambda)} \, .$$ 
In view of  the definition of the Fourier transform on the 
group $G$, we thus have
$$\displaylines{\quad \| {\mathcal F}(f)(\lambda,\nu)h_{\alpha,\eta(\lambda)}\|_{L^2( \mathfrak p_\lambda)}^2 = \Big(1+ |\nu|^2+\zeta(\alpha,\lambda)\Big)^{-2M} \hfill\cr\hfill \times\!\! \int_{\mathfrak p_\lambda}
\! \biggl( \int_{G} \!\!\! \big ((\mbox{Id}-\Delta_{G }\!\big)^M\! f(x))u^{\lambda,\nu}_{X(\lambda,x) }   h_{\alpha,\eta(\lambda)}(\xi) \,d\mu(x)
\overline{\int_{G} \!\!\!\big((\mbox{Id}\!-\!\Delta_{ G }\!\big)^M\! f(x'))u^{\lambda,\nu} _{X(\lambda,x') }  h_{\alpha,\eta(\lambda)}(\xi)\, d\mu(x')\!}\biggr)d\xi\, .}
$$
Now, by    Fubini's theorem, we get
$$\displaylines{\quad \| {\mathcal F}(f)(\lambda,\nu)h_{\alpha,\eta(\lambda)}\|_{L^2( \mathfrak p_\lambda)}^2 = \Big(1+|\nu|^2+\zeta(\alpha,\lambda)\Big)^{-2M}
\hfill\cr\hfill\times \int_{G} \int_{G}\! (\mbox{Id}-\Delta_{G }\!)^M f(x)
  \overline{ (\mbox{Id}-\Delta_{G }\!)^Mf(x')} (
 u^{\lambda,\nu}_{X(\lambda,x) }  h_{\alpha,\eta(\lambda)}\!\mid\! u^{\lambda,\nu}_{X(\lambda,x') }  h_{\alpha,\eta(\lambda)})_{L^2(\mathfrak p_\lambda)}\,  d\mu(x) \, d\mu(x')\, .\quad}$$
 Since the operators~$u^{\lambda,\nu}_{X(\lambda,x) } $ and~$u^{\lambda,\nu}_{X(\lambda,x') } $ are
unitary on~$\mathfrak p_\lambda$ and the family~$(h_{\alpha,\eta(\lambda)})_{\alpha\in{\mathbb N}^d}$ is a
Hilbert basis of~$\mathfrak p_\lambda$, we deduce that 
$$\| {\mathcal F}(f)(\lambda,\nu)h_{\alpha,\eta(\lambda)}\|_{L^2( \mathfrak p_\lambda)} \leq  \Big(1+ |\nu|^2+\zeta(\alpha,\lambda)\Big)^{-M} \|(\mbox{Id}- \Delta_{ G })^M f\|_{L^1(G)}\, .$$
Because
$$ {\rm Card}\Bigl(\Bigl\{\alpha\in{\mathbb N}^d\,/\, |\alpha|=m\Bigr\}\Bigr)=\binom{m+d-1}{m} \leq C(m+1)^{d-1}\, ,
$$
 this ensures that
$$
\longformule{
\sum_{\al \in {\mathbb N}^d} \int_{\lambda\in\Lambda}\int_{\nu\in\mathfrak r^*_\lambda}  \|{\mathcal F}(f)(\lambda,\nu) h_{\alpha,\eta(\lambda)}\|_{L^2( \mathfrak p_\lambda)}  |{\mbox {Pf}} (\lambda) |\,  d\nu\,d\lambda\, \lesssim \|(\mbox{Id}- \Delta_{G })^M f\|_{L^1(G)}
}
{{}
\times \sum _m (m+1)^{d-1}\int_{\lambda\in\Lambda}\int_{\nu\in\mathfrak r^*_\lambda}\Big(1+|\nu|^2+\zeta(\alpha,\lambda)\Big)^{-M}  |{\mbox {Pf}} (\lambda) |\,  d\nu\,d\lambda\,  .} 
$$ 
Hence taking  $M=M_1+M_2,$ with $\ds M_2 > \frac k 2$ implies that 
$$
\longformule{
\sum_{\al \in {\mathbb N}^d} \int_{\lambda\in\Lambda}\int_{\nu\in\mathfrak r^*_\lambda}  \|{\mathcal F}(f)(\lambda,\nu) h_{\alpha,\eta(\lambda)}\|_{L^2( \mathfrak p_\lambda)}  |{\mbox {Pf}} (\lambda) |\,  d\nu\,d\lambda\, \lesssim \|(\mbox{Id}- \Delta_{G })^M f\|_{L^1(G)}
}
{
{}
\times\sum _m (m+1)^{d-1}\int_{\lambda\in\Lambda}\Big(1+\zeta(\alpha,\lambda)\Big)^{-M_1}  |{\mbox {Pf}} (\lambda) |\,  d\lambda\,  .} 
$$ 
 Noticing that~$\zeta(\alpha,\lambda)=0$ if and only if $\lambda=0$ and using the homogeneity of degree~$1$ of $\zeta$, yields that there exists $c>0$ such that $\zeta(\alpha, \lambda  )\geq c \, m\, | \lambda |$. Therefore, we can end the  proof of~(\ref{hyp:inv}) by    choosing~$M_1$ large enough and performing the change of variable~$\mu = m \, \lam $ in each term of the above series. 

\medskip
\noindent
Proposition~\ref{inversioninS} is proved.
 \end{proof}
%%%%%%%

%%%%%%%%%%%%%%%%%%%%%%%%%%%%%%%%%%%%%%%%%%%%%%%%%%


\begin{thebibliography}{99}

\bibitem{AP} J.-P. Anker and V. Pierfelice, 
Nonlinear Schr‚àö‚àÇdinger equation on real hyperbolic spaces, 
{\it Annales de l'Institut Henri Poincar\'e, Analyse non lin\'eaire}, {\bf 26}  (2009), pages 1853-1869.

\bibitem{A} R. Anton, Strichartz inequalities for Lipschitz metrics on manifolds and nonlinear Schr√∂dinger equation on domains, {\it Bulletin de la Soci\'et\'e Math\'ematique de France}, {\bf 136}  (2008), pages 27-65. 


\bibitem{bch} H. Bahouri and J.-Y. Chemin, \'Equations d'ondes quasilin\'eaires
et in\'egalit\'es de Strichartz, {\it American Journal of
Mathematics}, {\bf 121}  (1999), pages 1337-1377.

\bibitem{bch2} H. Bahouri and J.-Y. Chemin, Microlocal analysis,
bilinear estimates and cubic quasilinear wave equation, {\it
Ast\'erisque, Bulletin de la Soci{\'e}t{\'e}
Math{\'e}\-ma\-ti\-que de France}, (2003), pages 93-142.

\bibitem{bcd} H. Bahouri, J.-Y. Chemin and R. Danchin,  Fourier Analysis and Nonlinear
Partial Differential Equations, {\it Grundlehren der mathematischen Wissenschaften, Springer}, (2011).

 \bibitem{bg}  H. Bahouri et I. Gallagher, Paraproduit sur le
groupe de Heisenberg et applications, {\it  Revista Matematica
  Iberoamericana}, {\bf 17} (2001), pages 69-105.

\bibitem{bgx} H. Bahouri, P. G{\'e}rard et C.-J. Xu,  Espaces de Besov et estimations de
Strichartz g{\'e}n{\'e}ra\-lis{\'e}es sur le groupe de Heisenberg,
{\it  Journal d'Analyse Math{\'e}matique}, {\bf  82} (2000), pages
93-118.


\bibitem{bfg} H. Bahouri,  C. Fermanian and I. Gallagher, Phase-space analysis and pseudodifferential calculus on the Heisenberg group,  {\it Ast\'erisque, Bulletin de la Soci{\'e}t{\'e}
Math{\'e}\-ma\-ti\-que de France}, {\bf  340}, (2012).

\bibitem{bfg2} H. Bahouri,  C. Fermanian and I. Gallagher, Refined Inequalities on Graded Lie Groups,  {\it Notes aux Comptes-Rendus de
l'Acad{\'e}mie  des Sciences de Paris}, {\bf 350}  S{\'e}rie I, (2012),
pages~393-397.

\bibitem{banica} V. Banica, 
The nonlinear Schr\"odinger equation on hyperbolic space, 
{\it Communication in  Partial Differential Equations}, {\bf 32} (2007), pages 1643-1677.

\bibitem{bd} V. Banica and T.  Duyckaerts, 
Weighted Strichartz estimates for radial Schr\"odinger equation on noncompact manifolds, 
{\it Dynamics of Partial Differential Equations}, {\bf 4} (2007), pages 335-359. 

\bibitem{brenner1} P. Brenner,  On $L^{p}-L^{p'}$
estimates for the wave-equation, {\it Mathematische
 Zeitschrift}, {\bf 145} (1975), pages 251-254.

\bibitem{bgt} N. Burq,  P. G\'erard and N.  Tzvetkov, 
Strichartz inequalities and the nonlinear Schr√∂dinger equation on compact manifolds, {\it American Journal of
Mathematics}, {\bf 126}  (2004), pages 569-605.


\bibitem{blp} N. Burq, G. Lebeau and F. Planchon, Global existence for energy critical waves in 3-D domains, {\it Journal of the American Mathematical Society}, {\bf 21}  (2008), pages 831-845.


\bibitem{crs} P. Ciatti, F. Ricci and M.  Sundari, Uncertainty inequalities on stratified nilpotent groups,  {\it Bulletin of Kerala Mathematics Association},  Special Issue, (2007), pages 53-72.


\bibitem{corwingreenleaf} L.-J. Corwin and F.-P. Greenleaf, Representations of nilpotent Lie groups and their applications, Part 1: Basic theory and examples,  {\it Cambridge studies in advanced Mathematics},  {\bf 18}  (1990), Cambridge university Press.

\bibitem{hiero}  M. Del Hierro,
Dispersive and Strichartz estimates on H-type groups, {\it Studia Math}, {\bf
169}   (2005), pages 1-20.

\bibitem{folland}  G. B. Folland,    Harmonic Analysis in Phase Space   {\it Annals of Mathematics Studies, Princeton University Press}, (1989).

\bibitem{follandstein} G. B. Folland and E. Stein,  Hardy Spaces on Homogeneous Groups,  {\it Annals of Mathematics Studies,  Princeton University Press}, (1982).

 \bibitem{furioli2} G. Furioli, C. Melzi and A. Veneruso,  Strichartz inequalities for the wave equation with the full Laplacian on the Heisenberg group, {\it Canadian Journal of Mathematics}, {\bf 59}  (2007), pages 1301-1322.

\bibitem{ginibrevelo} J. Ginibre and G. Velo, Generalized Strichartz inequalities for the wave equations, {\it  Journal of  Functional Analysis}, {\bf  133}   (1995), pages  50-68.


\bibitem{ilp} O. Ivanovici, G. Lebeau and F. Planchon, Dispersion for the wave equation inside strictly convex domains I: the Friedlander model case,  {\it  Annals of Mathematics},  {\bf  180}   (2014), pages  323-380.
\bibitem{keeltao} M. Keel and T. Tao,  Endpoint Strichartz estimates, {\it American Journal of  Mathematics},  {\bf 120} (1998), pages
   955-980.
   
   \bibitem{kapitanski}
L. Kapitanski: Some generalization of the Strichartz-Brenner
inequality, {\it Lenin\-grad Mathematical Journal}, {\bf 1} (1990),
pages 693-721.

\bibitem{kr}   S. Klainerman and I. Rodnianski, 
Rough solutions of the Einstein-vacuum equations,
{\it Annals  of Mathematics}, {\bf 161}  (2005),  pages 1143-1193. 

\bibitem{koranyi2} A. Kor{\'a}nyi,  Geometric properties of {H}eisenberg-type groups, {\it Advances in Mathematics},  {\bf 56} (1985), pages
   28-38.

\bibitem{Ludwig} J. Ludwig,  Dual topology of diamond groups, {\it Journal f‚àö¬∫r die reine und angewandte Mathematik}, {\bf 467}  (1995), pages 67-87.

\bibitem{MooreWolf} C. Moore and J. Wolf, Square integrable representations of nilpotent Lie groups, {\it Transactions of the American Mathematical Society}, {\bf 185} (1973), pages~445-462.  

\bibitem{MR} D. M\"uller  and F. Ricci, Solvability for a class of doubly characteristic differential operators on two-step nilpotent groups, {\it Annals of  Mathematics}, {\bf  143}  (1996),   pages 1-49.

\bibitem{P1} H. Pecher,  $L^p$ - Absch\"atzungen und klassische
L\"osungen f\"ur nichtlineare Wellengleichungen, {\it
Mathematische Zeit\-schrift}, {\bf 150} (1976), pages 159-183.

\bibitem{Poguntke} D.  Poguntke,  Unitary representations of diamond groups, {\it Mathematica Scandinavica}, {\bf  85} (1999),  pages 121-160.

\bibitem{segal} I. E. Segal, Space-time decay for solutions of wave equations, {\it Advances in Mathematics}, {\bf  22}  (1976),   pages 304-311.


\bibitem{smith1}
H. Smith, A parametrix construction for wave equation
with~$C^{1,1}$ coefficients, {\it Annales de l'Institut Fourier},
{\bf 48} (1998), pages 797-835.

\bibitem{ss}
H. Smith and  C. Sogge, On the critical semilinear wave equation outside convex obstacles, {\it Journal of the American Mathematical Society}, {\bf 8}  (1995), pages 879-916.


\bibitem{st} H. Smith and D.Tataru, 
Sharp local well-posedness results for the nonlinear wave equation, 
{\it Annals  of Mathematics}, {\bf 162}  (2005), pages  291-366. 

\bibitem{stein} E.M. Stein:
Harmonic Analysis: Real-Variable Methods, Orthogonality, Oscillatory integrals, Princeton University Press, (1993).
\bibitem{stein3} E.M. Stein:
Oscillatory integrals in Fourier analysis, {\it Beijing Lectures in
Harmonic Analysis}, Princeton University Press, (1986), pages
307-355.



\bibitem{steinweiss} E. Stein and G. Weiss,  Introduction to Fourier analysis on euclidean spaces, {\it Princeton Mathematical Series,  Princeton University Press, Princeton, N.J.}, {\bf  32} (1971).

\bibitem{strichartz}
R. Strichartz, Restriction Fourier transform of quadratic surfaces
and decay of solutions of the wave equations, {\it Duke
Mathematical Journal}, {\bf 44}  (1977), pages 705-714.


\bibitem{tao} T. Tao,  Nonlinear dispersive equations. Local and global analysis,  {\it CBMS Regional Conference Series in Mathematics, American Mathematical Society, Providence, RI}, {\bf 106} (2006).

\bibitem{tataru} D.Tataru, 
Strichartz estimates in the hyperbolic space and global existence for the semilinear wave equation,
{\it Transactions of the   American  Mathematical society}, {\bf 353}  (2001), pages 795-807.



\end{thebibliography}
\end{document}